\RequirePackage{tikz}
\documentclass[sn-mathphys-num]{sn-jnl}

\usepackage{graphicx}%
\usepackage{multirow}%
\usepackage{amsmath,amssymb,amsfonts}%
\usepackage{amsthm}%
\usepackage{mathrsfs}%
\usepackage[title]{appendix}%
\usepackage{xcolor}%
\usepackage{textcomp}%
\usepackage{manyfoot}%
\usepackage{booktabs}%
\usepackage{algorithm}%
\usepackage{algorithmicx}%
\usepackage{algpseudocode}%
\usepackage{listings}%
\usepackage{todonotes}
\usepackage[utf8]{inputenc}
\usepackage[T1]{fontenc}
\usepackage{geometry}
\usepackage{microtype}
\usepackage{subcaption}
\usepackage{bbm}
\usepackage{pdflscape}
\usepackage{bm}
\usepackage{url}
\usepackage{multirow}
\usepackage{setspace}
\usepackage{mathtools} 
\usepackage{xcolor}
\usepackage{xspace}

\usepackage{booktabs}
\usepackage{csquotes}
\usepackage{soul}
\usepackage{environ} 
\usepackage{trimclip}
\usepackage{forest}
\forestset{
	declare toks={elo}{}, 
	anchors/.style={grow'=90, anchor=#1,child anchor=#1,parent anchor=#1},
	dot/.style={tikz+={\fill (.child anchor) circle[radius=#1];}},
	dot/.default=2pt,
	decision edge label/.style n args=3{
		edge label/.expanded={node[midway,auto=#1,anchor=#2,\forestoption{elo}]{\strut$\unexpanded{#3}$}}
	},
	decision/.style={if n=1
		{decision edge label={left}{east}{#1}}
		{decision edge label={right}\label{key}{west}{#1}}
	},
	decision tree/.style={
		for tree={grow'=90,
			s sep=2.5pt, l=0pt, l sep =0.5pt, outer sep =-1.5pt,
			if n children=0{anchors=west}{
				if n=1{anchors=west}{anchors=west}},
			math content,
		},
		anchors=west, outer sep=-1.5pt,
		dot=2pt, for descendants=dot,
		delay={for descendants={split option={content}{;}{content,decision}}},
	},
	rooted tree/.style={
		for tree={
			grow'=90,
			parent anchor=center,
			child anchor=center,
			s sep=2.5pt,
			l sep =1pt,
			if level=0{
				baseline
			}{},
			delay={
				if content={*}{
					content=,
					append={[]}
				}{}
			}
		},
		before typesetting nodes={
			for tree={
				circle,
				fill,
				minimum width=3pt,
				inner sep=0pt,
				child anchor=center,
			},
		},
		before computing xy={
			for tree={
				l=5pt,
			}
		}
	}
}
\DeclareDocumentCommand\ct{o}{\Forest{decision tree [#1]}}
\DeclareDocumentCommand\rt{o}{\Forest{rooted tree [#1]}}
\usepackage{tikz}
\usepackage{tkz-euclide} 
\usetikzlibrary{patterns}
\newlength{\tickl}    
\tikzset{axes/.style={thick,-latex}}
\tikzset{lineplot/.style={thick}}
\tikzset{arrow/.style={thick,-latex}} 
\tikzset{thick arrow/.style={ultra thick,-latex}} 
\tikzset{grid lines/.style={very thin,gray!30}}	
\tikzset{point/.style={radius=2pt}}
\tikzset{help line/.style={black,thin,dashed}} 

\makeatletter
\newsavebox{\measure@tikzpicture}
\NewEnviron{scaletikzpicturetowidth}[1]{%
	\def\tikz@width{#1}%
	\begin{lrbox}{\measure@tikzpicture}%
		\BODY
	\end{lrbox}%
	\pgfmathparse{#1/\wd\measure@tikzpicture}%
	\BODY
}
\makeatother
\usetikzlibrary{decorations.markings}
\usepackage{pgfplots}

\tikzset{mylabel/.style  args={at #1 #2  with #3}{
		postaction={decorate,
			decoration={
				markings,
				mark= at position #1
				with  \node [#2] {#3};
} } } }

\renewcommand{\b}[1]{\mathbf #1}
\newcommand{\mean}[1]{	\{\!\{ #1 \}\!\}}
\newcommand{\jump}[1]{	[\![ #1 ]\!]}

\newcommand{\ie}{i.\,e.\ }
\newcommand{\eg}{e.\,g.}
\newcommand{\dt}{\Delta t}
\newcommand{\dx}{\Delta x}
\newcommand{\dtdx}{\frac{\dt}{\dx}}

\newcommand{\bsigma}{\boldsymbol{\sigma}}

\renewcommand{\O}{\mathcal O}

\newcommand{\R}{\mathbb R}

\newcommand{\dd}{\mathrm{d}}

\newcommand{\tm}{\subseteq}
\newcommand{\abs}[1]{\lvert #1\rvert}

\newcommand{\tend}{t_\mathrm{end}}

\DeclareMathOperator{\diag}{diag}

\renewcommand{\tt}[2]{\resizebox{#1 cm}{!}{\marginbox*{0pt -0.25\height pt 0pt 0pt}{#2}}}
\newcommand{\bspi}{\boldsymbol{\pi}}

\newcommand{\byk}{\b u^{(k)}}

\definecolor{colorA}{rgb}{0,0.447,0.741}
\definecolor{colorB}{rgb}{0.85,0.325,0.098}
\definecolor{colorE}{rgb}{0.929,0.694,0.125}
\definecolor{colorF}{rgb}{0.494,0.184,0.56}
\definecolor{colorD}{rgb}{0.466,0.674,0.188}
\definecolor{colorC}{rgb}{0.301,0.745,0.933}
\definecolor{colorG}{rgb}{0.635,0.078,0.184}

\theoremstyle{thmstyleone}%
\newtheorem{theorem}{Theorem}
\newtheorem{lemma}[theorem]{Lemma}
\newtheorem{proposition}[theorem]{Proposition}%

\theoremstyle{thmstyletwo}%
\newtheorem{remark}{Remark}%

\theoremstyle{thmstylethree}%
\newtheorem{definition}{Definition}%

\begin{document}

\title[Short Title]{Flux-Balanced Patankar-type Schemes for the Compressible Euler Equations}


\author*[1,2]{\fnm{Thomas} \sur{Izgin}}\email{izgin@mathematik.uni-kassel.de,  thomas\_izgin@brown.edu, ORCID:~0000-0003-3235-210X}

\author[1]{\fnm{Andreas} \sur{Meister}}\email{meister@mathematik.uni-kassel.de}
%
\author[2]{\fnm{Chi-Wang} \sur{Shu}}\email{chi-wang\_shu@brown.edu, ORCID:~0000-0001-7720-9564}
\author[3]{\fnm{Davide} \sur{Torlo}}\email{davide.torlo@uniroma1.it, ORCID:~0000-0001-5106-1629}

\affil*[1]{\orgdiv{Institute of Mathematics}, \orgname{University of Kassel}, \orgaddress{\street{Heinrich-Plett-Str. 40}, \city{Kassel}, \postcode{34132}, \state{Hessen}, \country{Germany}}}

\affil[2]{\orgdiv{Division of Applied Mathematics}, \orgname{Brown University}, \orgaddress{\street{182 George Street}, \city{Providence}, \postcode{02906}, \state{Rhode Island}, \country{USA}}}

\affil[3]{\orgdiv{Department of Mathematics}, \orgname{Università di Roma La Sapienza}, \orgaddress{\street{Piazzale Aldo Moro 5}, \city{Rome}, \postcode{00185 }, \country{Italy}}}



\abstract{Positivity preservation of key physical quantities in the context of fluid flows, such as density and internal energy, is an essential property of a numerical scheme as otherwise the solution lacks physical relevance and has a not well-defined equation of state. One time integration technique that is capable of preserving the positivity of quantities for every time step size is the Patankar-trick and its variants. However, in the context of the Euler equations of gas dynamics, we wonder whether the Patankar-trick should be applied to the density and total energy equations or only to one of them.
In this work, we discuss one drawback of the schemes when blindly applied to every positive conserved variable and additionally point out how to overcome the issue by balancing the involved numerical fluxes correctly. 
To illustrate our findings, we investigate modified Patankar--Runge--Kutta (MPRK) schemes in the context of the compressible Euler equations with and without stiff  source terms. 
We discover that it is beneficial to only apply the Patankar-trick in the density equation and to balance the remaining numerical fluxes consistently rather than applying the trick also to the energy equation. 
This leads also to the preservation of contact discontinuities. 
We perform numerical experiments to demonstrate that the accuracy of the methods is maintained while the performance of our approach is superior to the traditional application of MPRK schemes.}





\maketitle

\section{Introduction}
For many natural phenomena, including the flow of fluids such as gases and water, analytic solutions are rarely derivable, and hence, numerical methods are essential for approximating solutions. However, these methods must preserve critical physical properties -- positivity (\eg, of density and pressure), conservation, and entropy stability -- to ensure both physical relevance and numerical robustness. 
In particular, failing to maintain positivity (\eg, negative mass or pressure) yields nonphysical solutions, as well as meaningless equations of state, rendering the simulation invalid \cite{Sandu2001,SSPMPRK2}. 
Moreover, neglecting conservation properties also result in physically incorrect behavior of the numerical approximation \cite{lax_systems_1960}. Often, this requires to evolve conserved quantities, leaving other variables, like pressure in the compressible Euler equations, to be derived using an equation of state (EoS). This makes their control especially challenging during the numerical approximation.
There are existing and ongoing works concerning numerical positivity, \eg, based on Rusanov-like fluxes for the Euler equations \cite{PS1996,CGMPT2023}, by using subcell limiters \cite{RPG2022,GD2021}, bound preserving filters \cite{Dzanic2024a,Dzanic2024b}, geometric projection \cite{WS2023}, or modified time stepping schemes \cite{EG2022,EG2023}.  However, none of these approaches yields an unconditionally positive method, meaning that there are potentially severe restrictions on the time step size to ensure positivity. 
To guarantee positivity unconditionally while achieving high accuracy, numerical schemes must be nonlinear \cite{sandu2001positive}, \ie the iterates are described by a nonlinear function even if the PDE is linear. 
Overall, several strategies have been developed to address the challenges related to unconditional positivity-preservation:
\begin{enumerate}
	\item \textit{Clipping techniques}, which forcibly set negative values to zero, either result in a mass-shifting optimization problem or otherwise compromise conservation, and, up-to-date, lack a proof of stability \cite{BIM2022}.
	\item \textit{Fully implicit, nonlinear methods} \cite{HR2020,ricchiuto2011habilitation} can enforce positivity but require costly iterative solvers, which may fail to converge (to a positive solution), and thus, still produce nonphysical results.
	\item	\textit{Patankar-type methods} represent a family of explicit or linearly implicit yet nonlinear schemes, which are unconditionally positive \cite{Patankar1980,BDM2003,MCD2020,KM18,AKM2020}.   
\end{enumerate}
The main idea behind Patankar-type schemes is to modify an existing time-stepping method by introducing nonlinear weights in such a way that the resulting numerical scheme becomes unconditionally positive. 
The primary challenge lies in designing these weights so that the modified scheme preserves the accuracy of the original (baseline) method. 
This nonlinear modification is achieved using the so-called \emph{Patankar-trick} \cite{Patankar1980}, which gives this family of methods its name.
A notable example is the incorporation of modified Patankar (MP) weights into classical Runge–Kutta (RK) schemes, leading to the development of modified Patankar–Runge–Kutta (MPRK) methods \cite{BDM2003,KM2018b}, which in addition to being unconditionally positive, are also conservative. 
Motivated by their strong numerical performance, the Patankar-trick has since been successfully extended to a variety of time integration frameworks, including strong-stability-preserving Runge–Kutta (SSPRK) methods \cite{SSPMPRK2,SSPMPRK3}, arbitrary high-order Deferred Correction (DeC) schemes \cite{MPDeC}, and linear multistep methods \cite{IMPV2025}. 
The resulting modified schemes all belong to the broader Patankar-type family, which themselves can be recast as non-standard additive Runge--Kutta (NSARK) methods, see \cite{NSARK, IzginThesis}.
Despite their advantages, analytical results for Patankar-type schemes have been difficult to derive due to their nonlinear structure. 
It is only recently, in the PhD thesis \cite{IzginThesis}, that rigorous results on Lyapunov stability \cite{IKM2022b} and high-order accuracy conditions \cite{NSARK} were developed. 
It is worth mentioning that this Lyapunov stability investigation is the first natural step when searching for reasonable well-balanced \cite{Ciallella2022} schemes as it ensures that initial errors around a constant steady state do not accumulate during the course of the simulation. Indeed, it is the only nonlinear stability analysis currently available for these nonlinear schemes.\newline
Despite progress, controlling the positivity of quantities defined via the EoS (like pressure or internal energy) remains a challenging task. Existing workarounds (e.g., taking absolute values or modifying parameters post hoc) lack reliability and physical grounding \cite{SSPMPRK2}. Using algebraic manipulations \cite{kuzmin2005algebraic} of the fluxes is a viable solution, but it requires an intrusive modification of existing codes. 
Further approaches consider giving up strict conservation of certain quantities (\eg, momentum or total energy) in favor of evolving directly primitive variables such as pressure and may lead to easier control on such variables, making use of some correction terms where conservation is needed \cite{abgrall2018high,gaburro2024discontinuous}.

In this work, we aim at making progress towards the design of unconditionally positive schemes for balance laws when conservative variables are evolved. In particular, we propose to apply the MP trick only to the density equation and to balance the remaining numerical fluxes consistently. This has the first effect of preserving contact discontinuities with constant pressure and velocity, which is a desirable property for the Euler equations. Moreover, we notice that during the simulations, it allows to maintain positivity of the pressure for larger values of the CFL with respect to classical MP methods, outperforming explicit methods and the traditional application of MPRK schemes.

The paper is organized as follows. We recall the compressible Euler equations with a single as well as with three reacting species in Section~\ref{sec:tests} and present the classical MPRK schemes in Section~\ref{sec:MPRK}. Then, we discuss how to apply the MPRK schemes in the presence of stiff source terms in Section~\ref{sec:schemes}, 
proposing a novel flux-balanced version of the momentum and total energy fluxes to guarantee the preservation of contact discontinuities and to improve the performance of the method with respect to classical MPRK schemes.
Finally, we perform numerical experiments in Section~\ref{sec:experiments} and summarize our work in Section~\ref{sec:summary}.

\section{Governing Equations}\label{sec:tests} 
In this work, we consider the compressible Euler equations of gas dynamics to describe the physical quantities of a single fluid, or of multi-species gas with chemical reactions. 
The Euler equations describe the evolution of density $\rho$ (or densities $\rho_i$), the momentum $\rho u$ and the total energy $\rho E$ of a gas. It is of paramount importance to preserve some physical properties of the gas dynamics such as positivity of the density and internal energy, and preservation of equilibrium states with constant velocity and pressure.  


\subsection{Euler equations}
The Euler equations and their reactive versions in one dimension can be written as hyperbolic balance laws of the type
\begin{equation}\label{eq:EulerEquations}
	\partial_t \b U(x,t)+ \partial_{x} \b F(\b U(x,t))  
	=\b S(\b U(x,t)), \qquad x \in \Omega \subset \mathbb R,\, t \in \mathbb R^+,
\end{equation}
where $\b U:\mathbb R \times \mathbb R^+ \to \mathbb R^d$ is the vector of the evolved quantities (when $\b S\equiv\b 0$ these are conserved quantities), $\b F:\mathbb R^d\to \mathbb R^d$ is the flux and $\b S:\mathbb R^d\to \mathbb R^d$ the source term.
The compressible Euler equations for a one species gas in a one dimensional domain is defined by 
\begin{equation}\label{eq:monoEuler}
	\b U=(\rho,\rho u,\rho E)^T,\, \b F(\b U)=(\rho u, \rho u^2+p,u(\rho E+p))^T \text{ and } \b S(\b U)  = (0,0,0)^T
\end{equation}
together with the equation of state (EoS) for ideal gas, \ie 
\[p= (\gamma-1)(\rho E - \tfrac{1}{2}\rho v^2),\]
where we use $\gamma= 1.4$ for air.

\subsection{Multi-Species Reactive Euler equations}
The three species reactive Euler equations, where two monoatomic gasses react with each other and there is a third diatomic gas, are defined \cite{WSYS2009,SSPMPRK2} by \eqref{eq:EulerEquations} with
\begin{align}\label{eq:multiEuler}
	&\b U=(
	\rho_1, \rho_2, \rho_3,	\rho u,	\rho E)^T,\, \b F(\b U)=(
	\rho_1u, \rho_2u, \rho_3u,
	\rho u^2+p,
	(\rho E+p)u)^T, \\
	\label{eq:source}
	&\b S(\b U)=\delta (2M_1\omega(\b U),-M_2\omega(\b U),0,0,0)^T, \,\text{with } \rho =\sum_{i=1}^3\rho_i,
\end{align}
where
\begin{align*}
	\omega(\b U) & =\left(k_f(T)\frac{\rho_2}{M_2}-k_b(T)\left(\frac{\rho_1}{M_1}\right)^2\right)\sum_{s=1}^3\frac{\rho_s}{M_s}, \\
	T&=\frac{p}{R\sum_{s=1}^3\frac{\rho_s}{M_s}}, & k_f(T)&=\frac{C}{T^2}\exp\left(-\frac{\hat E}{T}\right), \\ k_b(T)&=\frac{k_f(T)}{\exp(b_1+b_2\ln(z)+b_3z+b_4z^2+b_5z^3)},& z&=\frac{10^4}{T}.
\end{align*}
Here, we have defined the inverse of the reaction characteristic time $\delta$ and we will make it vary during the simulations to test very severe stiffness in the source for large $\delta$.
The involved constants are
\begin{align*}
	M_1 &=0.016,&	M_2 &=0.032,&  M_3&=0.028 &R&=287,\\ C&=2.9\cdot10^{17},& \hat E&=59750,& b_1&=2.85, &b_2&=0.988,\\ b_3&=-6.181,&
	b_4&= -0.023,& b_5&= -0.001,&h^0_1 &= 1.558\cdot 10^7. && &
\end{align*}
The equation of state (EoS)	reads
\begin{equation}\label{eq:EOS}
	p= \frac{\sum_{s=1}^3\frac{\rho_s}{M_s}}{\tfrac32 \frac{\rho_1}{M_1} +\tfrac52 \frac{\rho_2}{M_2} + \tfrac32 \frac{\rho_3}{M_3}}\left(\rho E - \rho_1 h_1^0 - \frac{m^2}{2\rho}\right)
\end{equation}
and the speed of sound, $c$, is given by
\[c(\b U) = \sqrt{\frac{\gamma p}{\rho}},\quad \gamma= 1 +\frac{p}{T\sum_{s=1}^3\rho_se_s'(T)}, \quad e_s(T)=\begin{cases}
	\tfrac32 \frac{RT\rho_s}{M_s}, & s=1,2,\\
	\tfrac52 \frac{RT\rho_s}{M_s}, & s=3.
\end{cases}\]
It is trivial to see that the total mass is conserved also by the source term as $\b S_1(\b U)= \delta 2M_1 \omega(\b U) = -\b S_2(\b U) = \delta M_2 \omega (\b U) $.

\subsection{Properties of Euler equations}
In both systems, the preservation of the positivity of the densities is fundamental not only to guarantee the physical relevance of the solution but also to ensure the well-posedness of the problem. In particular, the positivity of the density is required to guarantee that the speed of sound is real-valued and that the system is hyperbolic. 
Moreover, it is also crucial to preserve the positivity of the pressure in order to compute the speed of sound and ensure that the EoS is well-defined. 

Additionally, in absence of reactions, both systems admit a family of solutions with constant velocity and pressure, which are important to preserve in order to avoid spurious oscillations around these equilibria.
Indeed, when $u(x,t)=\bar u$ and $p(x,t)=\bar p$ for some constants $\bar u,\bar p>0$, the fluxes in \eqref{eq:monoEuler} and \eqref{eq:multiEuler} are constant for the variables $\rho u$ and $\rho E$, while the densities are described by transport equations with constant velocity $\bar u$. This is the case with a contact discontinuity solution, described by 
\begin{equation}\label{eq:contact_disc}
	\rho(t,x)=\begin{cases}
		\rho_L, & x<\bar u t,\\
		\rho_R, & x>\bar u t,
	\end{cases}\quad u(t,x)=\bar u, \quad p(t,x)=\bar p,
\end{equation}
which naturally appears even in simple Riemann problems. Given that most of the numerical methods for the Euler equations are based on the solution of Riemann problems, it is crucial to preserve such equilibria to avoid spurious oscillations around them.

In what follows, we will try to preserve such properties. The use of MPRK schemes will guarantee the positivity of the densities, while the preservation of the equilibria with constant velocity and pressure will be achieved by balancing the involved numerical fluxes correctly. The positivity of the pressure is less straightforward to guarantee, but we will show that the new flux-balanced version of MPRK schemes, introduced in Section~\ref{sec:schemes}, will require a less restrictive time step size to maintain the positivity of the pressure compared to the traditional application of MPRK schemes.

\section{Modified Patankar--Runge--Kutta schemes}\label{sec:MPRK}
We will now recall the definition of MPRK schemes and their main properties. We will then discuss how to apply them to the compressible Euler equations in the presence of non-periodic boundary conditions and source terms in Section~\ref{sec:schemes}.
Following \cite{IssuesMPRK,IR2023,IzginThesis,Ciallella2022}, MPRK schemes have been developed to solve the so-called production-destruction-rest systems (PDRS),
\begin{equation}
		u_i'(t)= r_i(\b u(t), t) + \sum_{j=1}^N (p_{ij}(\b u(t), t)-d_{ij}(\b u(t), t)),\quad \b u(0)=\b u^0\in \R^N_{>0}, \label{eq:PDRS}
	\end{equation}
where $p_{ij}(\b u(t),t), d_{ij}(\b u(t),t) \geq 0$ for all $\b u(t)>\b 0$, $t\geq 0$ for $i,j=1,\dotsc,N$. Here and in the following, vector inequalities should be understood in a pointwise way. The expressions $p_{ij}$ and $d_{ij}$ represent production and destruction terms, respectively, while $r_i$ refer to rest terms. In this notation, we require that the PDS part satisfies $p_{ij}=d_{ji}$ as well as $p_{ii}=d_{ii}=0$ for $i,j=1,\dotsc,N$. For later use and in contrast to \cite{IssuesMPRK}, the rest terms are also split for $i=1,\dotsc,N$ according to
\begin{equation}
    r_i(\b u(t), t) = r_i^p(\b u(t), t) - r_i^d(\b u(t), t)\label{eq:rp_rd}
\end{equation}
with $r_i^p,r_i^d\geq 0$ for $t\geq 0$ and  $\b u(t)\geq \b 0$. 
Note that $r_i^p$ and $r_i^d$ can always be constructed, see \cite{IR2023} or \eqref{eq:rPrD_energy} for an illustrative example.

The PDRS \eqref{eq:PDRS} is said to be \textit{positive} if any solution $\b u\colon I\to \R^N$ with $I\tm \R_{\geq 0}$ and $\b u(0)>\b 0$ satisfies $\b u(t)>\b 0$ for all $t\in I$, and \textit{conservative} if $r_i=0$ for all $i=1,\dotsc,N$. In this context, MPRK schemes are defined as follows.

\begin{definition}\label{def:MPRKdefn}
Given an explicit $s$-stage RK method described by a non-negative Butcher array, \ie $\b A,\b b,\b c \geq\b 0$, we define the corresponding MPRK scheme applied to the PDRS \eqref{eq:PDRS}, \eqref{eq:rp_rd} by
\begin{equation}\label{eq:MPRK_PDRS}
    \begin{aligned}
        u_i^{(k)}=& u_i^n + \dt\sum_{\nu=1}^{k-1}a_{k\nu}\left( r_i^p(\b u^{(\nu)}, t_n + c_\nu\dt) + \sum_{j=1}^N  p_{ij}(\b u^{(\nu)}, t_n + c_\nu\dt)\frac{u_j^{(k)}}{\pi_j^{(k)}}\right. \\
        & \left. - \left( r_i^d(\b u^{(\nu)}, t_n + c_\nu\dt)+\sum_{j=1}^Nd_{ij}(\b u^{(\nu)}, t_n + c_\nu\dt) \right)\frac{u_i^{(k)}}{\pi_i^{(k)}}\right),\quad k=1,\dotsc,s,\\
        u_i^{n+1}=& u_i^n + \dt\sum_{k=1}^{s}b_{k}\left( r_i^p(\b u^{(k)}, t_n + c_k\dt) + \sum_{j=1}^N  p_{ij}(\b u^{(k)}, t_n + c_k\dt)\frac{u_j^{n+1}}{\sigma_j}\right. \\
        & \left. - \left(r_i^d(\b u^{(k)}, t_n + c_k\dt) + \sum_{j=1}^Nd_{ij}(\b u^{(k)}, t_n + c_k\dt)\right)\frac{u_i^{n+1}}{\sigma_i}\right)
    \end{aligned}
\end{equation}
for $i=1,\dotsc,N$, where $\pi_i^{(k)},\sigma_i$ are the so-called \emph{Patankar-weight denominators} (PWDs), which are required to be positive for any $\dt\geq 0$. Additionally, $\pi_i^{(k)}=\pi_i^{(k)}(\b u^n,\b u^{(1)}, \dots, \b u^{(k-1)})$ is independent of $\mathbf u^{(k)}$, and $\sigma_i=\sigma_i(\b u^n,\b u^{(1)}, \dots,\b u^{(s)})$ is independent of $\mathbf u^{n+1}$.
\end{definition}
As described in \cite{KM18,IzginThesis}, MPRK schemes can be written in matrix notation that, for the PDRS presented above, read as follows.
\begin{remark}
    In matrix notation, \eqref{eq:MPRK_PDRS} can be rewritten as
    \begin{equation}\label{eq:MPRK_PDRS_matrix}
        \begin{aligned}
            \b M^{(k)}\byk&= \b u^n+\dt \sum_{\nu=1}^{k-1}a_{k\nu} \b r^p(\b u^{(\nu)}, t_n + c_\nu\dt),\quad k=1,\dotsc,s, \\
            \b M\b u^{n+1}&= \b u^n + \dt\sum_{k=1}^{s}b_{k} \b r^p(\b u^{(k)}, t_n + c_k\dt),
        \end{aligned}
    \end{equation}
    where $\b r^p=(r_1^p,\dotsc, r_N^p)^T$, $\b M^{(k)}=(m^{(k)}_{ij})_{1\leq i,j\leq N}$, and $\b M=(m_{ij})_{1\leq i,j\leq N}$ with
    \begin{equation}\label{eq:M_stage}
        \begin{aligned}
            m^{(k)}_{ii}&= 1+ \dt \sum_{\nu=1}^{k-1}a_{k\nu}\left(r_i^d(\b u^{(\nu)}, t_n + c_\nu\dt) + \sum_{j=1}^Nd_{ij}(\b u^{(\nu)}, t_n + c_\nu\dt) \right)\frac{1}{\pi_i^{(k)}}, \\
            m^{(k)}_{ij}&= -\dt \sum_{\nu=1}^{k-1}a_{k\nu}p_{ij}(\b u^{(\nu)}, t_n + c_\nu\dt)\frac{1}{\pi_j^{(k)}}, \quad i\neq j
        \end{aligned}
    \end{equation}
    as well as
    \begin{equation}\label{eq:M}
        \begin{aligned}
            m_{ii}&= 1+ \dt \sum_{k=1}^{s}b_k\left(r_i^d(\byk, t_n + c_k\dt)+\sum_{j=1}^N d_{ij}(\byk, t_n + c_k\dt) \right)\frac{1}{\sigma_i}, \\
            m_{ij}&= -\dt \sum_{k=1}^{s}b_kp_{ij}(\byk, t_n + c_k\dt)\frac{1}{\sigma_j}, \quad i\neq j.
        \end{aligned}
    \end{equation}
\end{remark}
The unconditional positivity of these schemes is then proved by showing that $\b M^{(k)}$ and $\b M$ in \eqref{eq:M_stage} and \eqref{eq:M} are $M$-matrices, \ie have non-negative inverses \cite{KM18}. Since Definition~\ref{def:MPRKdefn} does not determines the PWDs, these can be chosen to ensure the order of accuracy of the method. 

Below, we give the specific choices used in this work, however, we note that there is a comprehensive theory for the order conditions \cite{KM18,KM18Order3, MPDeC,IzginThesis,NSARK}. 

\paragraph{First-Order MPRK Scheme (MPE)}
Note that $\b u^{(1)}=\b u^n$, so $\bspi^{(1)}$ does not have to be defined.
The first-order scheme is based on forward Euler, which means that only $\bsigma$ needs to be defined. To obtain first order of accuracy, we use $\bsigma= \b u^n$, resulting in the MPE scheme from \cite{BDM2003}.

\paragraph{Second-Order MPRK Scheme (MPHeun)}
Although there is a one-parameter family of second-order MPRK schemes \cite{KM18}, we follow the suggestions from \cite{IssuesMPRK} and we will use the scheme based on Heun's method, given by $\bspi^{(2)}=\b u^n$ and $\bsigma=\b u^{(2)}$.

\section{Modified Patankar Finite Volume Schemes}\label{sec:schemes}
We start explaining the application of MPRK methods in the simple case of the one species Euler equations \eqref{eq:EulerEquations}-\eqref{eq:monoEuler} and then we will generalize for the multispecies Euler equations \eqref{eq:EulerEquations}-\eqref{eq:multiEuler}. Afterwards, we will highlight an issue of the method when applied to the one species Euler equations \eqref{eq:EulerEquations}-\eqref{eq:monoEuler} with a contact discontinuity \eqref{eq:cont_disc}, and how to overcome this problem.

Given $\Omega = [a,b]$ and equispaced points $x_{j+\frac12}=a+j\dx$ with $\dx=\frac{b-a}{N}$, $j=0,\dotsc,N$, we consider the cells $\Omega_i=[x_{i-\frac12},x_{i+\frac12}]$ with $i=1,\dotsc,N$. Along this work, we will use classical Finite Volume methods \cite{leveque2002finite} describing the solution  by its cell averages $\b U_i \approx \frac1\dx\int_{\Omega_i} \b U (x) \dd x$. The evolution of these values will depend on interface fluxes that are defined starting from the neighboring values.
In this work, we use the local Lax--Friedrichs (LLF) numerical flux, which, using the notation
\begin{equation}\label{eq:mean_jump}
\mean{\b g (\b U)}_{i+\frac12} \coloneqq \frac{\b g(\b U_i) + \b g(\b U_{i+1})}{2}, \quad \jump{\b g(\b U)}_{i+\frac12} \coloneqq \b g(\b U_{i+1})-\b g(\b U_i)
\end{equation} can be written as
\begin{equation}\label{eq:LF}
\begin{split}
    \hat{\b F}_{i+\frac12} &= \mean{\b F(\b U)}_{i+\frac12} -\frac{\alpha_{i+\frac12}}{2} \jump{\b U}_{i+\frac12},\\ \alpha_{i+\frac12} &= \max\{\abs{u_i} + c(\b U_i), \abs{u_{i+1}} + c(\b U_{i+1}), \abs{\mean{u}_{i+\frac12}} + c(\mean{\b U}_{i+\frac12})\}.
\end{split}
\end{equation}
This allows to write the semi-discrete form of \eqref{eq:EulerEquations} as
\begin{equation}\label{eq:EulerSemiDisc}
	\frac{\mathrm{d} \b U_i}{\mathrm{d} t}=\frac{1}{\dx}\left( \hat{\b F}_{i-\frac12}-\hat{\b F}_{i+\frac12}\right)+\b S(\b U_i),
\end{equation}
where $\b U_i=(U_{i,1},U_{i,2},U_{i,3})^T=((\rho)_i, (\rho u)_i,	(\rho E)_i)^T$ represent the cell averages of the physical quantities for each cell $\Omega_i$, and we applied the midpoint rule to the source term. 
In the following, we describe how to apply the MPE scheme to the semi-discrete form of Euler equations in order to guarantee positivity of the density and total energy. The generalization to higher-order schemes is then straightforward.

To illustrate the main idea, and for the sake of simplicity, let us first focus on the density components of \eqref{eq:EulerSemiDisc} for the Euler system \eqref{eq:EulerEquations}-\eqref{eq:monoEuler}, but the same technique will be applied to other positive variables. In the simulations, we will test also the version where both density and total energy are treated with the MP trick. Those version will be denoted with a -$\rho E$, e.g. MPE-$\rho E$. The density equation reads 
\begin{equation}\label{eq:EulerSemiDiscEnergy}
	\frac{\mathrm{d} \rho_{i}}{\mathrm{d} t}=\frac{1}{\dx}\left( \hat{F}_{i-\frac12,1}-\hat{F}_{i+\frac12,1}\right),\quad  i=1,\dotsc,N,
\end{equation}
as $S_1(\b U_i)=0$. 
First, note that due to the boundary conditions, the sum $\sum_{i=1}^N\left(\hat{F}_{i-\frac12,1}-\hat{F}_{i+\frac12,1}\right)= \hat{F}_{\frac12,1}-\hat{F}_{N+\frac12,1}$ does not vanish in general, so that \eqref{eq:EulerSemiDiscEnergy} must be viewed as a PDRS in order to apply the MPE scheme. Hence, we rewrite \eqref{eq:EulerSemiDiscEnergy} as a PDRS
\begin{equation}\label{eq:EulerSemiDiscEnergyPDRS}
	\frac{\mathrm{d} \rho_{i}}{\mathrm{d} t}=\frac{1}{\dx}\left(r_{i,1}^p-r_{i,1}^d+\sum_{j=1}^N\left(p_{i,j,1}-d_{i,j,1} \right)\right),\quad  i=1,\dotsc,N,
\end{equation}
where $r_{i,1}^p,r_{i,1}^d,p_{i,j,1},d_{i,j,1}\geq 0$, and $p_{i,j,1}=d_{j,i,1}$. In particular, we have
\begin{equation}\label{eq:Prod}
    p_{i+1,i,1}=d_{i,i+1,1}=\max\left\{0,\hat{F}_{i+\frac12,1}\right\}, \quad p_{i,i+1,1}=d_{i+1,i,1}=-\min\left\{0,\hat{F}_{i+\frac12,1}\right\}
\end{equation}
for $i=1,\dotsc,N-1,$
and $p_{i,j,1}=d_{j,i,1}=0$, if $\lvert i-j\rvert\neq1$.
Furthermore, the terms $\b r_1^p$ and $\b r_1^d$ are given by
\begin{equation}\label{eq:rPrD_energy}
	\begin{aligned}
		\b r^{p}_1&=\max\left\{0,\hat{F}_{\frac12,1}\right\}\b e_1+ \max\left\{0, -\hat{F}_{N+\frac12,1}\right\}\b e_N,\\
		\b r^{d}_1&=-\left(\min\left\{0,\hat{F}_{\frac12,1}\right\}\b e_1+ \min\left\{0, -\hat{F}_{N+\frac12,1}\right\}\b e_N\right)
	\end{aligned}
\end{equation}
with the unit vectors $\b e_1,\b e_N\in \R^N$. Now, the MPE scheme for the density reads 
\begin{equation}\label{eq:MPE_density}
	 \rho_{i}^{n+1}=\rho_{i}^n+\dtdx r^{p,n}_{i,1} + \dtdx\left( \sum_{j=1}^Np^n_{i,j,1}\frac{\rho_{j}^{n+1}}{\rho_{j}^n}-\left(r^{d,n}_{i,1}+ \sum_{j=1}^Nd^n_{i,j,1}\right)\frac{\rho_{i}^{n+1}}{\rho_{i}^n}\right),
\end{equation}
for $i=1,\dotsc,N$
and only differs from the forward Euler (FE) by the fractional weights.
We note that the MPE scheme is also first-order accurate when applied to a PDRS, see \cite{NSARK,IR2023} for more details. 
Moreover, the solution $\b u^{n+1}_1=( \rho^{n+1}_{1}, \dotsc,  \rho^{n+1}_{N})^T$ of \eqref{eq:MPE_density} can be obtained by solving the linear system $\b M_1\b u^{n+1}_1=\b u^n_1 + \dtdx \b r^{p,n}_1$, where  the matrix $\b M_1\in \R^{N\times N}$ is an M-matrix and therefore possesses a non-negative inverse. Furthermore, due to $p^n_{i,j,1}=d^n_{j,i,1}=0$ for $\lvert i- j\rvert \neq1$, the matrix $\b M_1$ is tridiagonal and sparse for large $N$. 
In particular, writing $\b M_1=(m_{i,j,1})_{1\leq i,j\leq N}$ as a special case of \eqref{eq:M}, we observe \[m_{i,i,1}=1+\dtdx\frac{r^{d,n}_{i,1}+ \sum_{j=1}^Nd^n_{i,j,1}}{U_{i,1}^n} \quad \text{ and }\quad m_{i,j,1}=-\dtdx\frac{p^n_{i,j,1}}{U_{j,1}^n},\quad   \lvert j-i\rvert =1\] and $m_{i,j,1} = 0$ otherwise. 

Similarly, this approach could also be used to guarantee the positivity of the total energy by solving the corresponding linear system $\b M_k \b u_k^{n+1}=\b u^n_k+\dtdx \b r^{p,n}_k$ for $k=3$ resulting in a total of two decoupled linear systems of size $N\times N$ that must be solved in each step, while the computation of the velocity remains unchanged, since there are no positivity constraints on the velocity.

\paragraph{MPE for Multi-species Euler}
In the case of multi-species Euler \eqref{eq:EulerEquations}-\eqref{eq:multiEuler}, one has to perform the same MPE method on each density equation and on the total energy, solving $\b M_k \b u_k^{n+1}=\b u^n_k+\dtdx \b r^{p,n}_k$ for $k=1,2,3,5$, while using the forward Euler scheme for the velocity equation.
One can define the vectors $\b u=(\b u_1^T, \b u_2^T, \b u_3^T,\b u_5^T)^T$ and $\b r^p=((\b r^p_1)^T,(\b r^p_2)^T,(\b r^p_3)^T,(\b r^p_5)^T)^T$ as well as the block diagonal matrix $\b M=\diag( \b M_1, \b M_2,\b M_3,\b M_5)$, so that $\b M \b u^{n+1}=\b u^n+\dtdx \b r^p$ represents the four decoupled linear systems that have to be solved. However, since source terms are present for the densities $\rho_1$ and $\rho_2$, we have to adjust our scheme as follows.

Consider the corresponding vectors $\b u_k=(U_{1,k},\dotsc, U_{N,k})^T=(\rho_{1,k},\dotsc, \rho_{N,k})^T$ with $k\in\{1,2\}$. The semi-discrete form \eqref{eq:EulerSemiDisc} of the reactive Euler equations for these two quantities reads
\begin{equation*}
	\begin{cases}
		\frac{\mathrm d}{\mathrm dt} \rho_{i,1} = \frac{\mathrm d}{\mathrm dt}U_{i,1}  = r_{i,1}^p-r_{i,1}^d+\sum_{j=1}^N\left(p_{i,j,1}-d_{i,j,1} \right) + S_1(\b U_i), & i=1,\dots, N,\\
		\frac{\mathrm d}{\mathrm dt}\rho_{i,2} =\frac{\mathrm d}{\mathrm dt}U_{i,2}  = r_{i,2}^p-r_{i,2}^d+\sum_{j=1}^N\left(p_{i,j,2}-d_{i,j,2} \right) + S_2(\b U_i), & i=1,\dots, N.
	\end{cases}
\end{equation*}
Next, we split each source term $S_k(\b U_i)=S^p_k(\b U_i)-S^d_k(\b U_i)$ into a production and a destruction part. 
In particular,
\[S^p_1(\b U_i)=\delta 2M_1k_f(T)\frac{U_{i,2}}{M_2}\sum_{s=1}^3\frac{U_{i,s}}{M_s}\geq 0  \quad \text{for $U_{i,s}\geq 0$}\] 
is the source production term and 
\[S^d_1(\b U_i)=\delta 2M_1k_b(T)\left(\frac{U_{i,1}}{M_1}\right)^2\sum_{s=1}^3\frac{U_{i,s}}{M_s}\geq 0 \quad \text{for $U_{i,s}\geq 0$}\] 
is the source destruction term of $S_1$. 
According to \eqref{eq:source} and $2M_1=M_2$, we observe $S_2^p(\b U_i)=S_1^d(\b U_i)$ and $S_2^d(\b U_i)=S_1^p(\b U_i)$. As a consequence, the MPE scheme for the first two densities reads for $i=1,\dots,N$
\begin{equation*}
	\begin{cases}
		\begin{split}
			\rho_{i,1}^{n+1}=U^{n+1}_{i,1} &= U^n_{i,1}+ \dtdx \left( r_{i,1}^{p,n}  +\sum_{j=1}^Np^n_{i,j,1}\frac{U^{n+1}_{j,1}}{U^n_{j,1}}-\left(r_{i,1}^{d,n}+\sum_{j=1}^Nd^n_{i,j,1} \right)\frac{U^{n+1}_{i,1}}{U^n_{i,1}} \right)\\
			&+ \dt \left(S_1^p(\b U^n_i)\frac{U^{n+1}_{i,2}}{U^n_{i,2}}-S_1^d(\b U^n_i)\frac{U^{n+1}_{i,1}}{U^n_{i,1}}\right)
		\end{split}\\
		\begin{split}
			\rho_{i,2}^{n+1}=U^{n+1}_{i,2} &= U^n_{i,2}+ \dtdx \left( r_{i,2}^{p,n}  +\sum_{j=1}^Np^n_{i,j,2}\frac{U^{n+1}_{j,2}}{U^n_{j,2}}-\left(r_{i,2}^{d,n}+\sum_{j=1}^Nd^n_{i,j,2} \right)\frac{U^{n+1}_{i,2}}{U^n_{i,2}} \right)\\
			&+ \dt \left(S_2^p(\b U^n_i)\frac{U^{n+1}_{i,1}}{U^n_{i,1}}-S_2^d(\b U^n_i)\frac{U^{n+1}_{i,2}}{U^n_{i,2}}\right)
		\end{split}\\
	\end{cases}
\end{equation*}
and hence, the full scheme for computing the densities and energy can be written as
\begin{equation}\label{eq:MPE_Vector}
	(\b M + \hat{\b S})\b u^{n+1}= \b u^n+\dtdx \b r^{p,n}
\end{equation}
with $\b M$ as defined above and $\hat{\b S}=\dt \begin{pmatrix}
\b S_{d,1} & \b S_{p,2} & \b 0&\b 0 \\
\b S_{p,1}	& \b S_{d,2} & \b 0&\b 0 \\
 \b 0 & \b 0 & \b 0& \b 0\\
	\b 0 & \b 0 & \b 0& \b 0
\end{pmatrix}$, where $\b S_{d,k}=\diag\left(\frac{S_k^d(\b U_1^n)}{U_{1,k}^n},\dotsc,\frac{S_k^d(\b U_N^n)}{U_{N,k}^n}\right) $ and $\b S_{p,k}=-\b S_{d,k}$. Note that the linear system \eqref{eq:MPE_Vector} is only coupled through the sparse matrix $\hat{\b S}$. 
\vspace{5mm}
\begin{lemma}
    The Matrix $\b B \coloneqq \b M + \hat{\b S}$ from \eqref{eq:MPE_Vector} is an $M$-matrix.
\end{lemma}
\begin{proof}
It is enough to restrict to the upper left $2N\times 2N$ block as the other blocks on the diagonal are M-matrices \cite{IzginThesis}.
Hence, fix $k\in \{1,2\}$ and observe that  $\b 1^T \b S_{d,k} =\Vert \b 1^T \b D_{p,k}\Vert_1$. Hence, since the diagonal of $\b B$ is positive and $\b M_k^T$ is diagonally dominant, we deduce that $\b B^T $ is also diagonally dominant along the same lines as in \cite[Lemma~2.8]{KM18}. Together with $b_{ij}\leq 0$ for $i\neq j$, we conclude that $\b B^T$ is an $M$-matrix, which means that the same holds for $\b B$.
\end{proof}
\subsection{Contact Discontinuities Preservation}
In a contact discontinuity, i.e., when $u(x)=\bar u$ and $p(x)=\bar p$ are constant values, whatever form takes the density, the density equation becomes a transport equation with velocity $\bar u$, while the momentum and energy conservation laws guarantee that velocity and pressure remain constant.

\begin{proposition}[Contact discontinuity preservation of explicit Euler LLF]\label{prop:LF}
    Consider the Euler equations \eqref{eq:EulerEquations}-\eqref{eq:monoEuler} with homogeneous Neumann or periodic boundary conditions, and apply the LLF scheme (explicit Euler method in time) together with the EoS $p=(\gamma-1)e$, where $e\coloneqq \rho E-0.5\rho u^2 $ is the internal energy. If the initial conditions contain only contact discontinuities, i.e., $u=\bar u$ and $p=\bar p$, then velocity and pressure stay constant for all time steps of the method.
\end{proposition}
\begin{proof}
 Since $u=\bar u$ and $p= \bar p$ are constant at the beginning, we observe that the LLF \eqref{eq:LF} for the momentum equation becomes 
\begin{equation}\label{eq:mom_LF}
    \hat{F}_{i+\frac12, 2} = \bar u\hat{F}_{i+\frac12, 1} + \bar p.
\end{equation}
This leads to the following update equation for the momentum
\begin{equation}
    \begin{split}
        \frac{(\rho u)^{n+1}_i - (\rho \bar{u})^{n}_i}{\Delta t} = -\frac{1}{\dx}\left( \hat{\b F}_{i+\frac12,2}^n - \hat{\b F}_{i-\frac12,2}^n\right)= -\bar u \frac{1}{\dx}\left( \hat{\b F}_{i+\frac12,1}^n - \hat{\b F}_{i-\frac12,1}^n\right),
    \end{split}
\end{equation}
which is clearly verified by $u^{n+1}_i=\bar{u}$ and by the LLF update of the density 
\begin{equation}\label{eq:rho_in_prop}
    \frac{\rho^{n+1}_i - \rho^{n}_i}{\Delta t} =-\frac{1}{\dx}\left( \hat{\b F}_{i+\frac12,1}^n - \hat{\b F}_{i-\frac12,1}^n \right).
\end{equation}
Note that it follows from \eqref{eq:mom_LF}-\eqref{eq:rho_in_prop} that $u$ will stay constant after one step of the LLF scheme.
Furthermore, since $p=\bar p$ we deduce from the EoS for ideal gases that the internal energy $e$ is also constant since $e= \rho E-0.5\rho u^2 = \frac{p}{\gamma-1}$. Hence, writing \[F_3(\b U)= u(\rho E + p)= u(e + 0.5\rho u^2 + p),\] we obtain
\begin{equation}\label{eq:energy_LF}
\hat{F}_{i+\frac12, 3} = 0.5\bar u^2\hat{F}_{i+\frac12, 1} + \bar u(\bar e + \bar p).
\end{equation}
From \eqref{eq:energy_LF}, we conclude that
\[(\rho E)_i^{n+1} = (\rho E)_i^n - \dtdx(\hat{F}_{i+\frac12, 3}-\hat{F}_{i-\frac12, 3})= (\rho E)_i^n + 0.5\bar u^2(\rho_i^{n+1}-\rho_i^n).\]
Thus,
\[e_i^{n+1}=(\rho E)_i^{n+1}-0.5\rho_i^{n+1} \bar u^2=(\rho E)_i^n-0.5\rho_i^n \bar u^2= \bar e, \]
which means that also pressure is constant in time. If the boundary conditions verify the same constraints $u=\bar{u}$ and $p=\bar p$, then the same arguments can be applied to boundary cells. 
\end{proof}

\subsubsection{Modified Patankar Euler LLF for Euler Equations}
However, the situation is different when MPRK schemes are applied, as the MP trick is never applied to the momentum equation. Hence, since $\rho$ is computed with an MPRK scheme, the $\rho u$ must be computed in a consistent manner to keep $u$ constant. 
Indeed, the main idea of the proof of Proposition~\ref{prop:LF} is that we can substitute the equation for the computation of $\rho$ into the equations for momentum and total energy. 
Hence, if we only modify the density fluxes $\hat F_{i+\frac12,1}$ using Patankar weights, all we have to do is to appropriately weigh the other numerical fluxes with the same Patankar-weights as for the density equation.
To that end, we note that \eqref{eq:mom_LF} and \eqref{eq:energy_LF} can be split into a part that contains the numerical flux of the density, which should be weighted exactly as in the density equation, while the other terms should be kept as in the explicit update of the original fluxes.

This strategy is in conflict with the weighting of the energy equation with an MP trick, so we could not guarantee the positivity of the total energy. Nevertheless, the pressure is the quantity whose positivity has to be guaranteed, and from the numerical experiments it seemed that the novel weighting due to the density MP trick facilitates the positivity of the pressure much more than using the MP trick on the energy equation. 
Hence, we will stick to this strategy and we will consider in future works the combination of MP trick on the energy equation and the preservation of contact discontinuities.
From here on, we focus on the ansatz of applying the MP trick to the density equation and weigh the other fluxes consistently.


Overall, we will split the numerical fluxes into a part that should be weighted and another one that should be left untouched. The first will have a superscript index $w$ and the latter will be indexed by $u$, \ie we split
$\hat{F}_{i+\frac12, k} = \hat{F}^w_{i+\frac12, k} + \hat{F}^u_{i+\frac12, k}$. According to \eqref{eq:mom_LF}, we define for the momentum equation
\begin{equation*}
     \hat{F}^w_{i+\frac12, 2} = \mean{\rho u^2}_{i+\frac12} - \frac{\alpha_{i+\frac12}}{2} \jump{\rho u}_{i+\frac12}, \quad  \hat{F}^u_{i+\frac12, 2} = \mean{p}_{i+\frac12}
\end{equation*}
using the jump $\jump{\cdot}_{i+\frac12}$ and mean $\mean{\cdot}_{i+\frac12}$ operators from \eqref{eq:mean_jump}.
From \eqref{eq:energy_LF}, the energy fluxes can be written in the form
\begin{equation*}
     \hat{F}^w_{i+\frac12, 3} = \mean{0.5 \rho u^3}_{i+\frac12} - \frac{\alpha_{i+\frac12}}{2} \jump{0.5\rho u^2}_{i+\frac12}, \quad  \hat{F}^u_{i+\frac12, 3} = \mean{u(e+p)}_{i+\frac12}- \frac{\alpha_{i+\frac12}}{2} \jump{e}_{i+\frac12}.
\end{equation*}
In what follows, we deduce the variant of the MPE scheme. The generalization to higher order is straightforward.
Now, recall that the production and destruction terms for the density are chosen analogously to \eqref{eq:Prod}. Hence, the weighting of  $\hat{F}^w_{i+\frac12, k}$ with $k\in \{2,3\}$ for momentum and energy will depend on the sign of $\hat{F}_{i+\frac12, 1}$. 
In particular, if $\hat{F}_{i+\frac12, 1}<0$ then this corresponds to a non-zero $p_{i,i+1,1}$ term with a weight of
\begin{equation}\label{eq:w}
    w^{n+1}_{i+1}\coloneqq \tfrac{\rho_{i+1}^{n+1}}{\rho_{i+1}^n},
\end{equation} 
and, if $\hat{F}_{i+\frac12, 1}>0$, the term corresponds to a destruction term $d_{i,i+1,1}$, which is weighted with $w^{n+1}_i$. The rest terms can be treated analogously. 
We define for $k=2,3$ and $i=1,\dotsc,N-1$ the auxiliary production and destruction terms by
\begin{equation}\label{eq:aux_prod}
    \mathfrak{p}_{i,i+1,k} \coloneqq \begin{cases}
    -\hat{F}^w_{i+\frac12, k},& \hat{F}_{i+\frac12, 1} <0,\\
    0,& \text{otherwise}
\end{cases},\quad \mathfrak{d}_{i,i+1,k} \coloneqq \begin{cases}
    \hat{F}^w_{i+\frac12, k},& \hat{F}_{i+\frac12, 1} \geq 0,\\
    0,& \text{otherwise}
\end{cases}, 
\end{equation}
and set $\mathfrak{p}_{i,j,k} = \mathfrak{d}_{j,i,k}$ for all $i,j=1,\dotsc,N$,
\begin{equation}
\mathfrak{r}^d_{1,k} \coloneqq \begin{cases}
-    \hat{F}^w_{\frac12, k},& \hat{F}_{\frac12, 1} <0,\\
    0,& \text{otherwise},
\end{cases} \mathfrak{r}^d_{2,k}=\dotsc=\mathfrak{r}^d_{N-1,k}=0, \mathfrak{r}^d_{N,k} \coloneqq \begin{cases}
    \hat{F}^w_{N+\frac12, k},& \hat{F}_{N+\frac12, 1} \geq 0,\\
    0,& \text{otherwise},
\end{cases}
\end{equation}
as well as $\mathfrak{r}^p_{i,k}\coloneqq  r_{i,k} + \mathfrak{r}^d_{i,k} + (\hat{F}^u_{i-\frac12, k}- \hat{F}^u_{i+\frac12, k})$.
Then, we can prove the following result.
\begin{theorem}
The scheme
\begin{equation}\label{eq:MPE_variant}
\begin{aligned}
	 U_{i,1}^{n+1}&=U_{i,1}^n+\dtdx r^{p,n}_{i,1} + \dtdx\left( \sum_{j=1}^Np^n_{i,j,1}w^{n+1}_j-\left(r^{d,n}_{i,1}+ \sum_{j=1}^Nd^n_{i,j,1}\right)w^{n+1}_i\right)\\
     	 U_{i,k}^{n+1}&=U_{i,k}^n+\dtdx \mathfrak r^{p,n}_{i,k} + \dtdx\left( \sum_{j=1}^N\mathfrak p^n_{i,j,k}w^{n+1}_j-\left(\mathfrak r^{d,n}_{i,k}+ \sum_{j=1}^N\mathfrak d^n_{i,j,k}\right)w^{n+1}_i\right), k=2,3,
     \end{aligned}
\end{equation}
with $w_i= U^{n+1}_{i,1}/U^n_{i,1}$ is first-order accurate and
unconditionally preserves the positivity of the density. Additionally, if $u$ and $p$ are initially constant, they remain constant at all times.
\end{theorem}
\begin{proof}
The positivity of the density is already known. Additionally, we chose the auxiliary production, destruction and rest terms in such a way that we can mimic the proof of Proposition~\ref{prop:LF} since we have 
\[p_{i,i\pm 1,1}= \mp \hat{F}_{i\pm\frac12, 1} \Longrightarrow \mathfrak p_{i,i\pm 1,k}= \mp \hat{F}^w_{i\pm\frac12, k},\]
and similar relations hold for the $\mathfrak d$ and $\mathfrak r^d$ terms. Finally, the consistent weighting allows us to substitute the weighted numerical density-fluxes into the other equations as we did in Proposition~\ref{prop:LF}. Finally, since the MPE scheme for the density equation is first-order accurate, we know from \cite{NSARK} that $w_i=1+\O(\dt)$. Substituting this relation into the momentum and total energy equation, we end up with the original LLF scheme plus an additional $\O(\dt^2)$ error.
\end{proof}
We want to remark that the weights added to the momentum and total energy equations are explicitly computed from the density update, which is the only equation where the MP trick is applied. Hence, this flux-balancing does not need further linear solves in contrast to classical MPRK schemes.

In the numerical results, we will see that this splitting and weighting of the fluxes of momentum and energy will also help to keep the pressure positive. 
It will not lead to a provably positivity preserving scheme for the pressure, but the simulations run with this split outperform by huge factors the scheme where the MP is applied to the density and total energy and explicit Euler is used for the momentum equation.

\subsubsection{Application to Multi-Species Euler}

Now, we can extend this result to multi-species Euler equations \eqref{eq:EulerEquations}, where the contact discontinuity preservation property is directly relevant only in the case without source terms. Nevertheless, we hope that the increased stability of the scheme, in particular in the positivity of the pressure, will also be transferred to the reactive case.

The methodology is very similar to the previous section: the densities equations are each of them treated with MP Euler, then we use the sum of the weights to weigh the part of the numerical fluxes of momentum and energy that should resemble a density numerical flux.
For $k=1,\dots,3$, we have that the three densities evolve as
\begin{equation}
    \frac{ \rho^{n+1}_{i,k}-\rho^{n}_{i,k}}{\Delta t} = \frac{1}{\Delta x} \left(r^{p,n}_{i,k} + \sum_{j=1}^N (p^n_{i,j,k} w_{j,k}^{n+1} - (r^{d,n}_{i,k}+d^n_{i,j,k})w_{i,k}^{n+1}) \right),
\end{equation}
with $w_{i,k}^{n+1} = \rho^{n+1}_{i,k}/\rho^{n}_{i,k}$.
Then, we take the sum of these equations over $k$, to obtain 
\begin{equation}
    \frac{ \rho^{n+1}_{i}-\rho^{n}_{i}}{\Delta t} = \frac{1}{\Delta x} \sum_{k=1}^3 \left(r^{p,n}_{i,k} + \sum_{j=1}^N (p_{i,j,k} w_{j,k}^{n+1} - (r^{d,n}_{i,k}+d_{i,j,k})w_{i,k}^{n+1}) \right).
\end{equation}
The update of $\rho$ is the quantity that should be factored out in the momentum and energy equations.
So, let us define for the momentum for $k=1,\dots,3$
\begin{equation*}
     \hat{F}^{w,k}_{i+\frac12, 4} = \mean{\rho_k u^2}_{i+\frac12} - \frac{\alpha_{i+\frac12}}{2} \jump{\rho_k u}_{i+\frac12}, \quad  
     \hat{F}^u_{i+\frac12, 4} = \mean{p}_{i+\frac12},
\end{equation*}
and for the energy, we define for $k=1,\dots,3$ fluxes
\begin{equation*}
     \hat{F}^{w,k}_{i+\frac12, 5} = \mean{0.5 \rho_k u^3}_{i+\frac12} - \frac{\alpha_{i+\frac12}}{2} \jump{0.5\rho_k u^2}_{i+\frac12}, \quad  
     \hat{F}^u_{i+\frac12, 5} = \mean{u(e+p)}_{i+\frac12}- \frac{\alpha_{i+\frac12}}{2} \jump{e}_{i+\frac12}.
\end{equation*}

We define for $\ell=4,5$, $k=1,\dots,3$ and $i=1,\dotsc,N-1$ the auxiliary production and destruction terms by
\begin{equation}\label{eq:aux_prod}
    \mathfrak{p}^k_{i,i+1,\ell} \coloneqq \begin{cases}
    -\hat{F}^{w,k}_{i+\frac12, \ell},& \hat{F}_{i+\frac12, k} <0,\\
    0,& \text{otherwise}
\end{cases},\quad \mathfrak{d}^k_{i,i+1,\ell} \coloneqq \begin{cases}
    \hat{F}^{w,k}_{i+\frac12, \ell},& \hat{F}_{i+\frac12, k} \geq 0,\\
    0,& \text{otherwise}
\end{cases}, 
\end{equation}
and set $\mathfrak{p}^k_{i,j,\ell} = \mathfrak{d}^k_{j,i,\ell}$ for all $i,j=1,\dotsc,N$,
\begin{equation}
\mathfrak{r}^{d,k}_{1,\ell} \coloneqq \begin{cases}
-    \hat{F}^{w,k}_{\frac12, \ell},& \hat{F}_{\frac12, k} <0,\\
    0,& \text{otherwise}
\end{cases}, \mathfrak{r}^{d,k}_{2,\ell}=\dotsc=\mathfrak{r}^{d,k}_{N-1,\ell}=0, \mathfrak{r}^{d,k}_{N,\ell} \coloneqq \begin{cases}
    \hat{F}^{w,k}_{N+\frac12, \ell},& \hat{F}_{N+\frac12, k} \geq 0,\\
    0,& \text{otherwise}
\end{cases},
\end{equation}
as well as $\mathfrak{r}^p_{i,\ell}\coloneqq  r_{i,\ell} + \mathfrak{r}^d_{i,\ell} + (\hat{F}^u_{i-\frac12, \ell}- \hat{F}^u_{i+\frac12, \ell})$.
Then, we can prove the following result.
\begin{theorem}
The scheme 
\begin{equation}\label{eq:MPE_variant_multispecies}
\begin{aligned}
	 U_{i,k}^{n+1}&=U_{i,k}^n+\dtdx\left( r^{p,n}_{i,k} +  \sum_{j=1}^Np^n_{i,j,k}w^{n+1}_{j,k}-\left(r^{d,n}_{i,1}+ \sum_{j=1}^Nd^n_{i,j,1}\right)w^{n+1}_{i,k}\right), \, k=1,\dots,3,\\
     	 U_{i,\ell}^{n+1}&=U_{i,\ell}^n+\dtdx  \left( \mathfrak r^{p,n}_{i,\ell} +  \sum_{k=1}^3 \left( \sum_{j=1}^N\mathfrak p^{k,n}_{i,j,\ell}w^{n+1}_{j,k}-\left(\mathfrak r^{d,k,n}_{i,\ell}+ \sum_{j=1}^N\mathfrak d^{k,n}_{i,j,\ell}\right)w^{n+1}_{i,k}\right)\right), \ell=4,5,
     \end{aligned}
\end{equation}
with $w_{i,k}= U^{n+1}_{i,k}/U^n_{i,k}$, discretizing the multi-species Euler without reaction source terms,
unconditionally preserves the positivity of the densities. Additionally, if  $u$ and $p$ are initially constant, they remain constant at all times. Moreover, the scheme is first-order accurate.
\end{theorem}
\begin{proof}
The proof is similar to the previous one, given the linearity of the weighted numerical fluxes on the densities. Concerning the accuracy, we know from \cite{NSARK} that $w_{j,k}^{n+1}=1+\O(\dt)$, and hence, the method has the same accuracy as the underlying explicit scheme.
\end{proof}
\begin{remark}
In the second order case, we know that $w_{j,k}^{n+1}=1+\O(\dt^2)$. However, there will be also weights for the second stage of the RK scheme, \ie $w_{j,k}^{(2)}=\frac{U_{j,k}^{(2)}}{\pi_{j,k}^{(2)}}= 1 + \O(\dt)$ for the densities \cite[Theorem~18]{NSARK}. Furthermore, the same theorem in \cite{NSARK} then guarantees that the second-order in time will be maintained for the other quantities, and any high-order space discretization can be applied to obtain an overall second-order of accuracy. In what follows, we use linear reconstruction with a minmod limiter to obtain second order away from extrema and non-smooth regions. 
\end{remark}

\section{Numerical Results}\label{sec:experiments}
In this section, we will validate the presented numerical method with balanced fluxes comparing it with the corresponding explicit schemes as well as with MP schemes applied to densities and with MP schemes applied to density and total energy equations. In particular, we will use the forward Euler (FE) method and Heun's method as explicit methods of first and second order, respectively. The corresponding MP schemes will be denoted by MPE and MPHeun, where we apply the MP trick to the densities equations only, MPE-$\rho E$ and MPHeun-$\rho E$ for the versions where both density and total energy are treated with the MP trick, and MPE-s and MPHeun-s for the versions where the MP trick is applied to the density and the fluxes of momentum and energy are weighted consistently with the density MP trick as described in the previous section.

\subsection{Choice of the CFL}
Most of the numerical experiments we will perform, undergo severe discontinuities and stiff source terms, which makes the choice of the time step crucial for the stability of the method.
Except for the first smooth test case, we will set the CFL number on a reference solution using the largest timestep we can use without obtaining negative densities or pressure.
This will not guarantee that the simulation gives qualitatively good results, hence, we will introduce a \textit{safety factor} that will premultiply the CFL, in order to obtain reliable simulations.
The comparison between the methods will be then mainly on their computational costs at, more or less, the same level of accuracy and reliability.

\subsection{Convergence Test}
We start demonstrating that the claimed convergence can be observed in smooth regions. 
For this test case, we consider the Euler equations \eqref{eq:EulerEquations}-\eqref{eq:monoEuler} on $[0,1]$, with the initial conditions
\begin{equation}\label{eq:IC_smooth}
	(\rho,u,p)(0,x)=(1,1, 1+ \cos(0.5\pi x)^4)
\end{equation} 
and final time $\tend = 0.03$. 
We compute the error on the density and the experimental order of convergence (EOC) for various number of cells for the first and second-order schemes using the $L_1$ and $L_2$ norms and a CFL number of 0.5 for all schemes. To that end, we define a reference solution using the same method with double the amount of cells, respectively.
As can be seen in Table~\ref{tab:eoc1}, the claimed accuracy is achieved. We note that the experimental order of convergence is reduced using the $L_2$ norm, see Table~\ref{tab:eoc2} as we use a minmod limiter that locally reduces the order to one at extrema. This is enough to impact the EOC when computed with the $L_2$ norm, but not the $L_1$ norm.

\begin{table}[!htbp]
	\caption{EOC for Euler equations \eqref{eq:EulerEquations}-\eqref{eq:monoEuler} with IC \eqref{eq:IC_smooth} using $L_1$ norm on $\rho$}\label{tab:eoc1}
	\centering
	\begin{tabular}{l|llll|llll}
		\hline
		& \multicolumn{4}{|c|}{First order}& \multicolumn{4}{c}{Second order}\\
		\hline
		$N$ \qquad\qquad& FE & MPE-$\rho E$ & MPE & MPE-s & Heun & MPHeun-$\rho E$ & MPHeun & MPHeun-s \\ 
		\hline 
		160  & 0.996 & 0.984 & 0.983 & 0.982 & 1.908 & 1.905 & 1.907 & 1.901 \\
		320  & 0.996 & 0.998 & 0.998 & 0.997 & 1.948 & 1.950 & 1.950 & 1.944 \\
		640  & 1.000 & 1.000 & 1.000 & 0.999 & 1.957 & 1.957 & 1.960 & 1.955 \\
		1280 & 1.000 & 1.001 & 1.001 & 1.000 & 1.966 & 1.966 & 1.968 & 1.965 \\
		2560 & 1.000 & 1.000 & 1.000 & 1.000 & 1.971 & 1.972 & 1.971 & 1.969 \\ \hline 
	\end{tabular}
\end{table}
\begin{table}[h]
	\centering
	\caption{EOC for Euler equations \eqref{eq:EulerEquations}-\eqref{eq:monoEuler} with IC \eqref{eq:IC_smooth}  using $L_2$ norm  on $\rho$}\label{tab:eoc2}
	\begin{tabular}{l|llll|llll}
		\hline
		& \multicolumn{4}{|c|}{First order}& \multicolumn{4}{c}{Second order}\\
		\hline
		$N$ \qquad\qquad& FE &  MPE-$\rho E$ &MPE & MPE-s & Heun & MPHeun-$\rho E$ & MPHeun & MPHeun-s \\ 
		\hline 
		160  & 0.989 & 0.974 & 0.973 & 0.970 & 1.715 & 1.708 & 1.709 & 1.709 \\
		320  & 0.988 & 0.990 & 0.989 & 0.987 & 1.716 & 1.721 & 1.722 & 1.710 \\
		640  & 0.992 & 0.994 & 0.994 & 0.992 & 1.704 & 1.714 & 1.713 & 1.700 \\
		1280 & 0.995 & 0.996 & 0.996 & 0.995 & 1.690 & 1.697 & 1.694 & 1.686 \\
		2560 & 0.997 & 0.997 & 0.997 & 0.996 & 1.678 & 1.683 & 1.680 & 1.675 \\
		\hline 
	\end{tabular}
\end{table}

\subsection{Reactive Euler Equations}
For the reactive Euler equations as a special case of the multi-species Euler equations, we consider the Riemann problem \begin{equation}
U_i(0,x)=\begin{cases}U_{i,L},& x<0,\\
U_{i,R},& x\geq 0,
\end{cases}
\end{equation}
defined by \cite{WSYS2009,SSPMPRK2}
 \[\begin{pmatrix} \rho_{1,L}\\ \rho_{2,L}\\  \rho_{3,L}\\v_L \\p_L\end{pmatrix}=\begin{pmatrix}5.251896311257205\cdot 10^{-5}\\ 3.748071704863518\cdot 10^{-5}\\ 2.962489471973072\cdot 10^{-4} \\0\\ 10^3\end{pmatrix}, \begin{pmatrix} \rho_{1,R}\\ \rho_{2,R}\\ \rho_{3,R}\\v_R \\p_R\end{pmatrix}=\begin{pmatrix} 8.341661837019181\cdot 10^{-8}\\ 9.45418692098664\cdot 10^{-11} \\2.748909430004963\cdot 10^{-7} \\0 \\1\end{pmatrix}\]
  together with homogeneous Neumann boundary conditions, and final time $\tend = 10^{-4}$. 
We will use $\delta =10^4$ as the reaction rate, which makes the problem very stiff.

We use $N=4000$ cells on the interval $[a,b]=[-1,1]$, \ie $\dx=\frac{2}{4000}$. 
The first-order schemes and second-order schemes produce similar approximations, respectively. 
However, due to the stiffness of the problem, the time step restrictions to have stability for the explicit schemes are very severe, see Table~\ref{tab:react_euler}. On the other hand, the classical MPRK schemes have weaker constraints on the time-step (around $\text{CFL}\sim 0.15$), while the flux-balanced version of Heun's method can run with incredibly high $\text{CFL}\sim 0.86$, leading to by far the best performance among the schemes. 

\begin{table}\caption{Comparison of numerical schemes applied to multi-species Euler equations \eqref{eq:EulerEquations}-\eqref{eq:multiEuler} with source \eqref{eq:source} with $\delta=10^4$ with $N=4000$ cells: stable CFL and runtime in seconds}\label{tab:react_euler}
\begin{tabular}{|llll|llll|}\hline
	\multicolumn{4}{|c|}{First order}&\multicolumn{4}{c|}{Second order} \\\hline
	Method & CFL & safety factor & Runtime & Method & CFL & safety factor & Runtime \\
	\hline 
	FE & 0.008 & 1 & 1494.76           & Heun & 0.008 & 1 &  5185.29 \\  
	MPE-$\rho E$ & 0.21 & 0.4 & 325.18     & MPHeun-$\rho E$ & 0.11 & 0.4 & 1960.50  \\ 
	MPE & 0.18 & 0.4 &  371.04    & MPHeun & 0.12 & 0.4 &   1671.91   \\  
	MPE-s & 0.19 & 0.4 & 470.39 			   & MPHeun-s  & 0.86 & 1 &  121.29 \\  
	\hline 
\end{tabular}

\end{table}
\begin{figure}
	\includegraphics[width=0.49\textwidth]{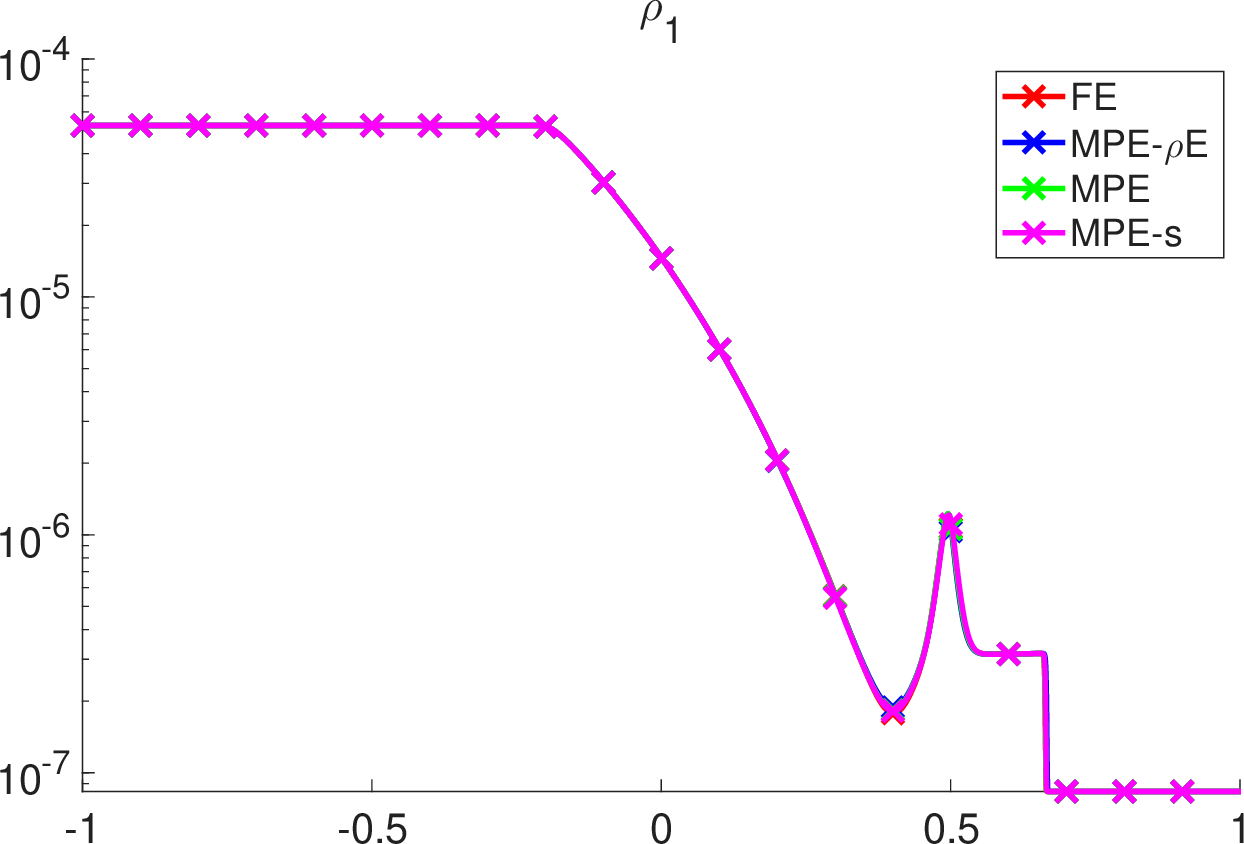}
\hfill
	\includegraphics[width=0.49\textwidth]{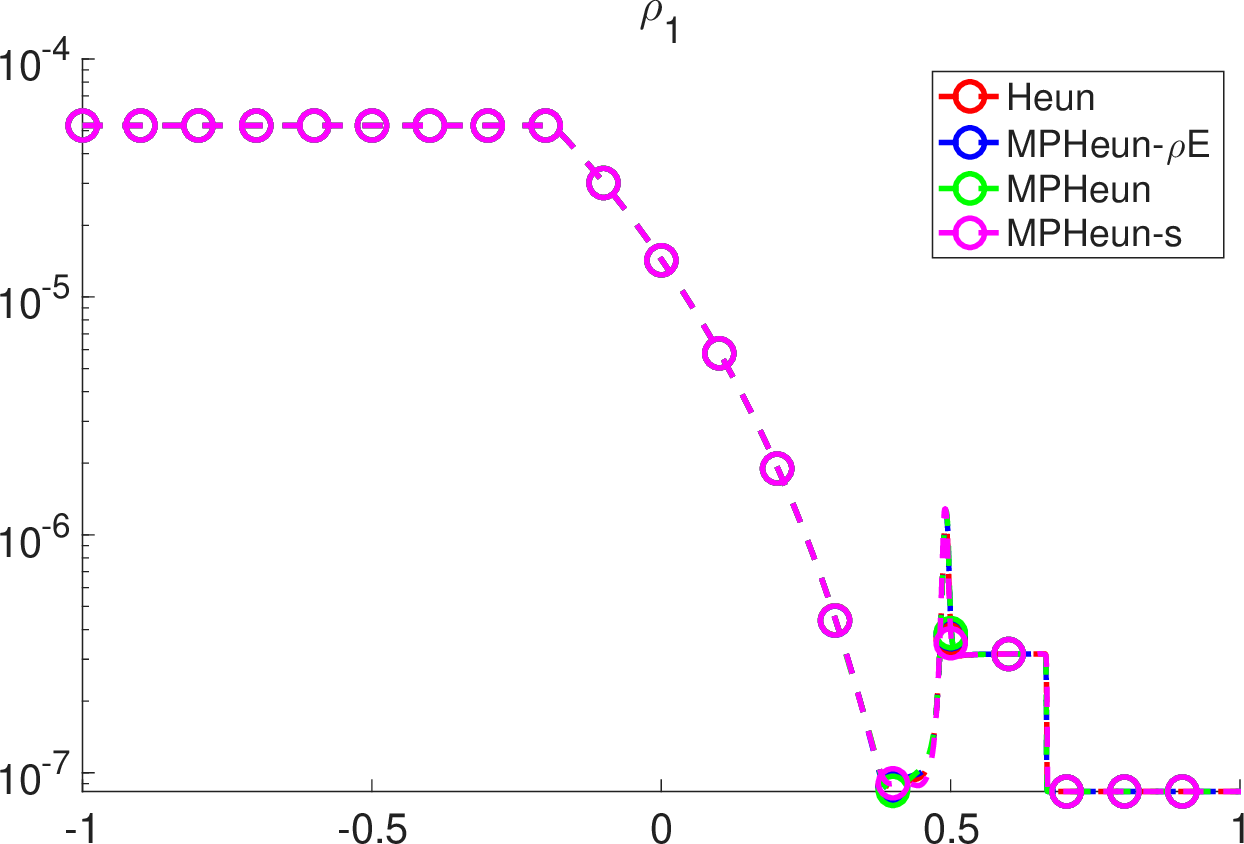}
\\[3mm]
	\includegraphics[width=0.49\textwidth]{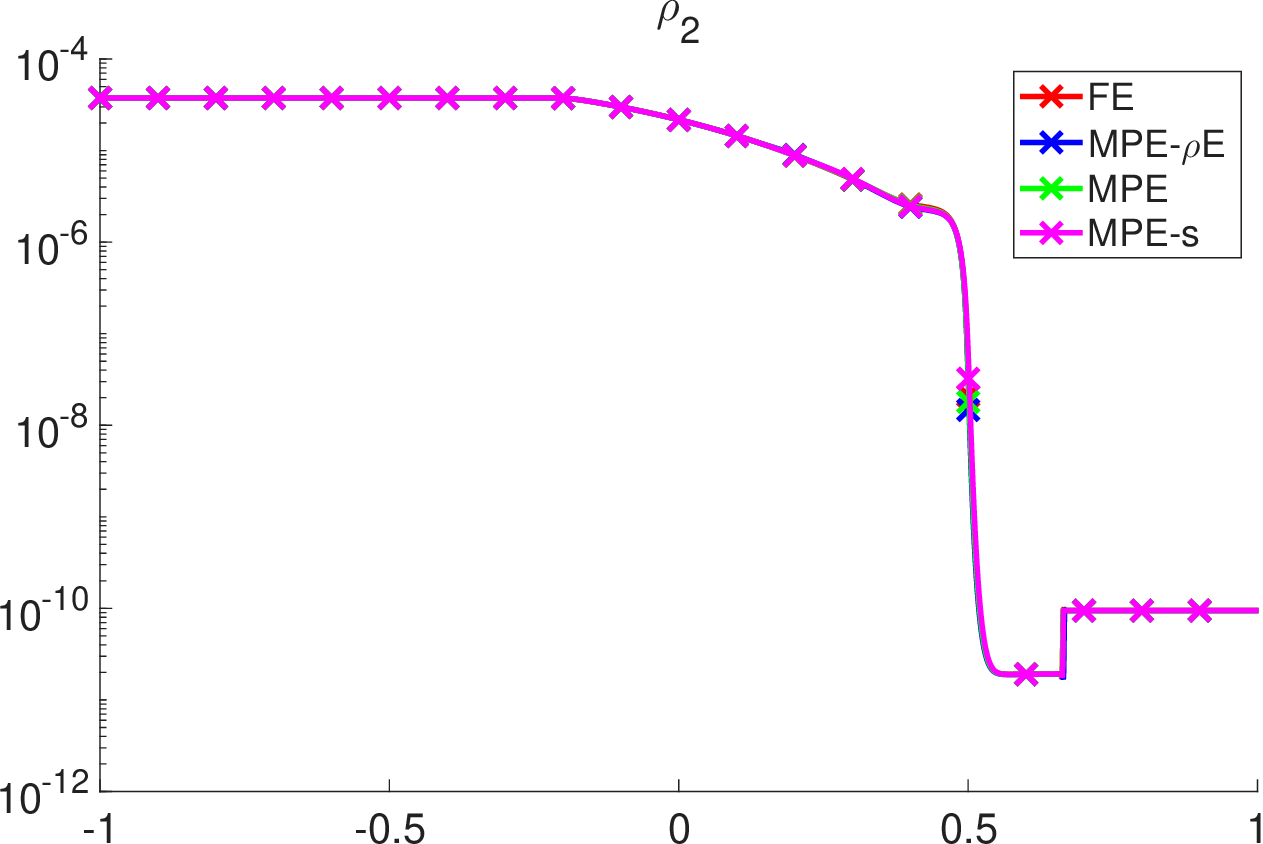}
\hfill
	\includegraphics[width=0.49\textwidth]{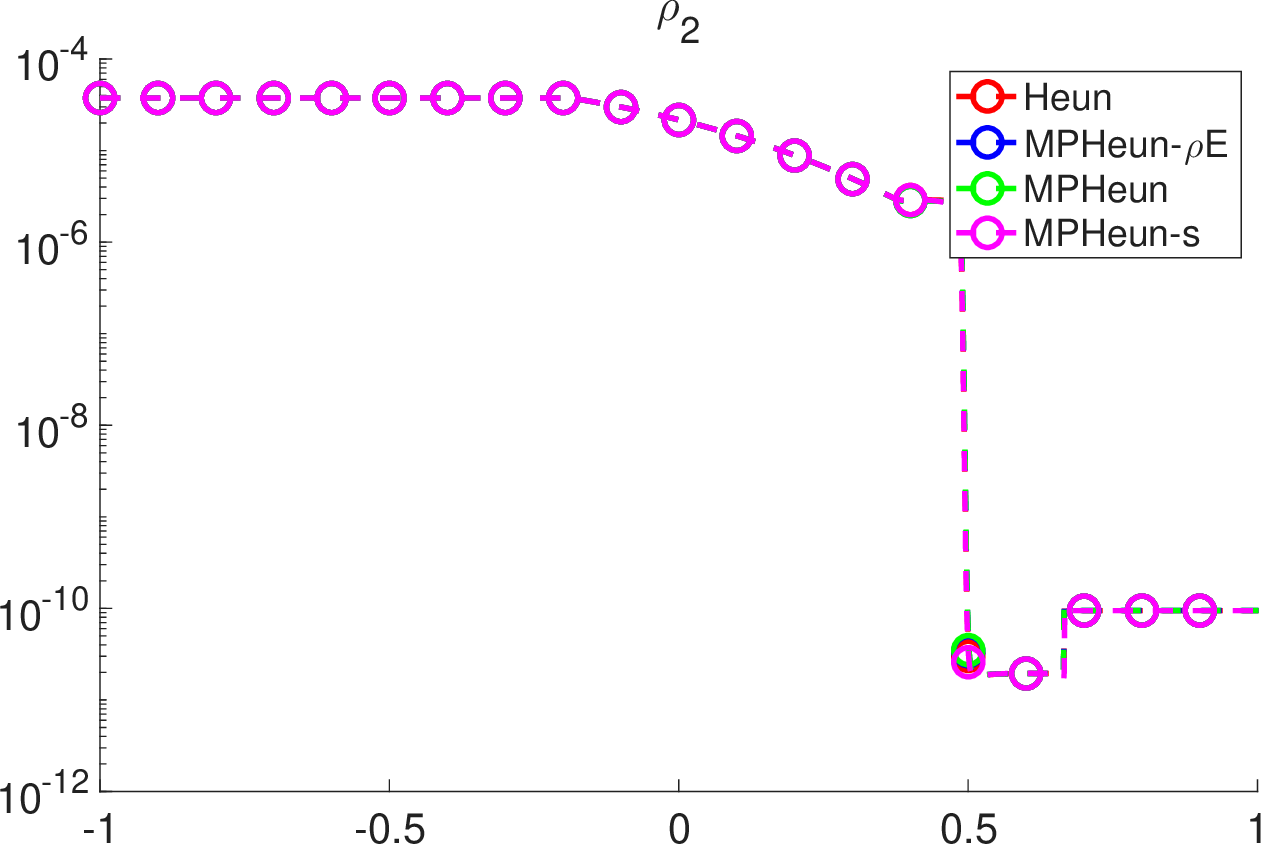}
\\[3mm]
	\includegraphics[width=0.49\textwidth]{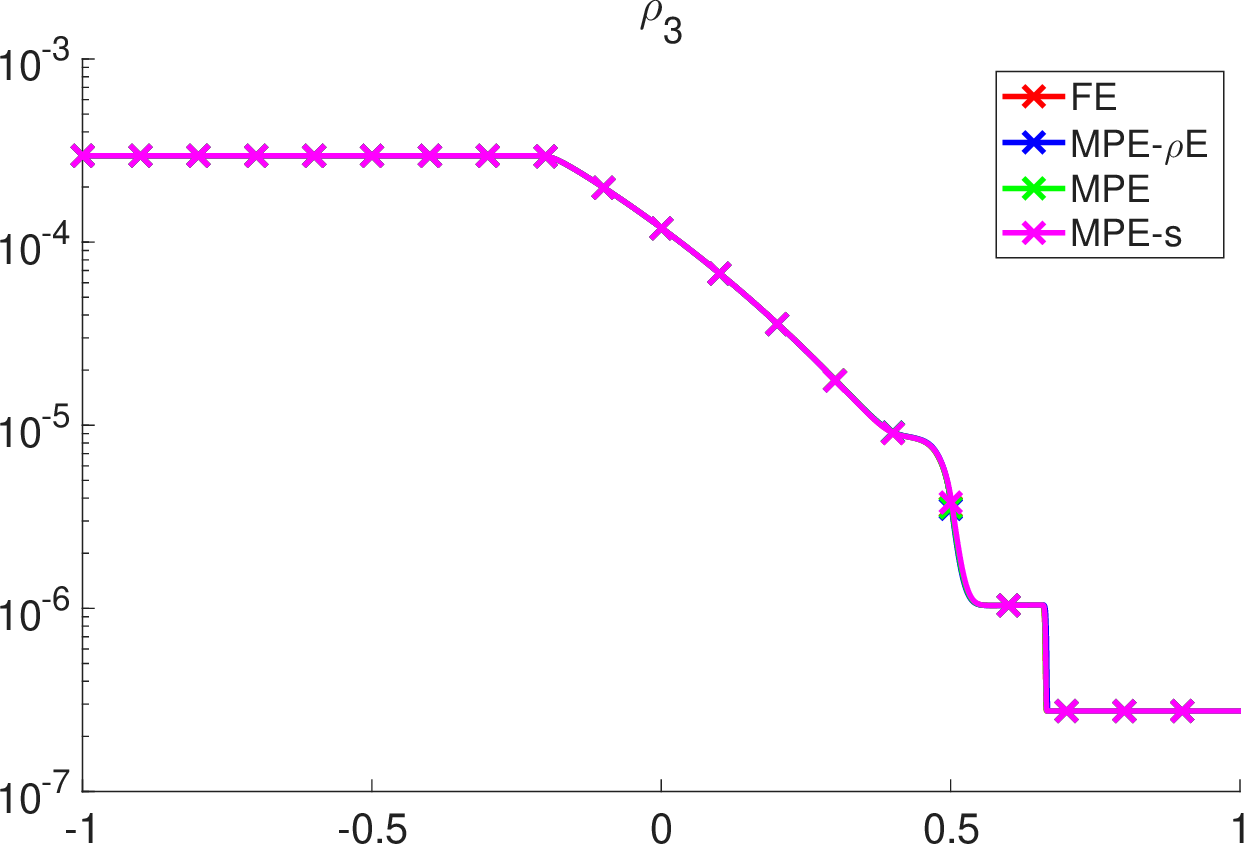}
\hfill
	\includegraphics[width=0.49\textwidth]{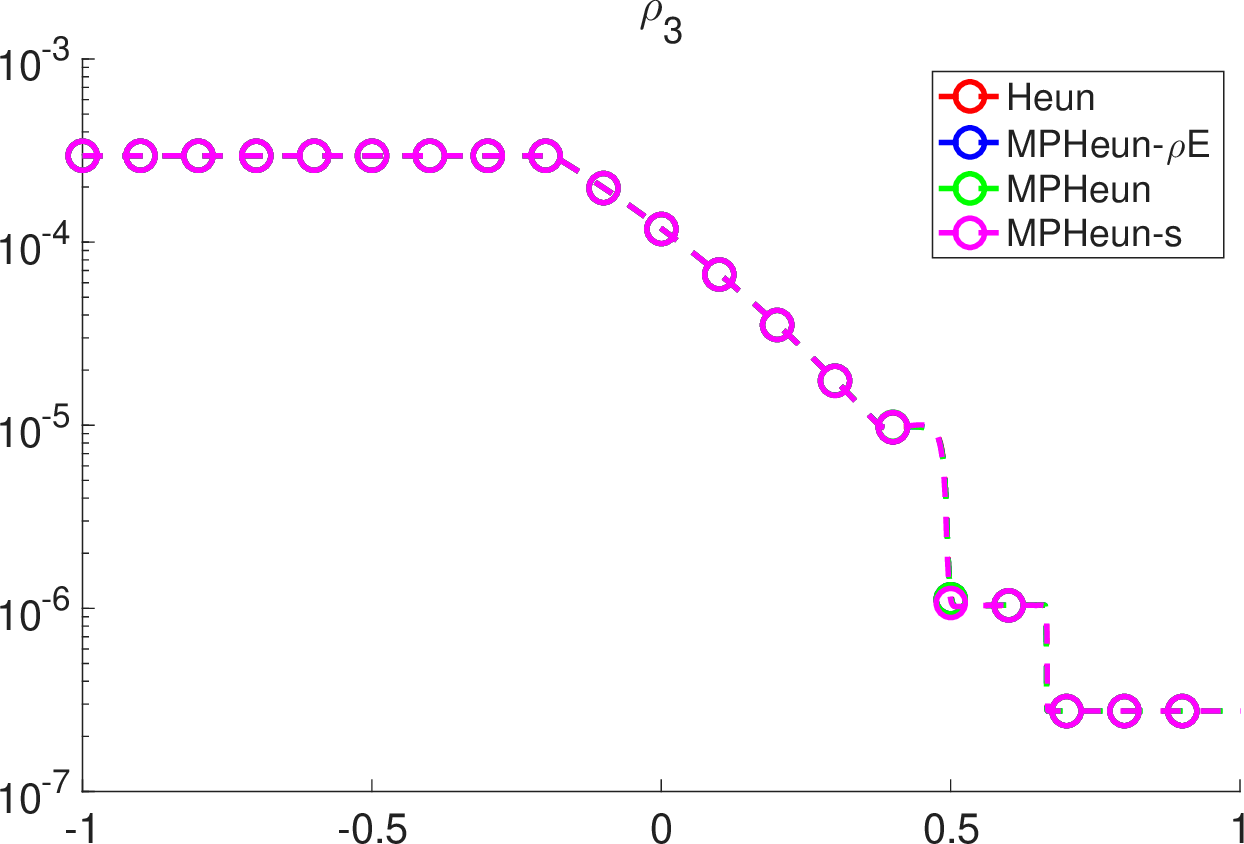}
\\[3mm]
	\includegraphics[width=0.49\textwidth]{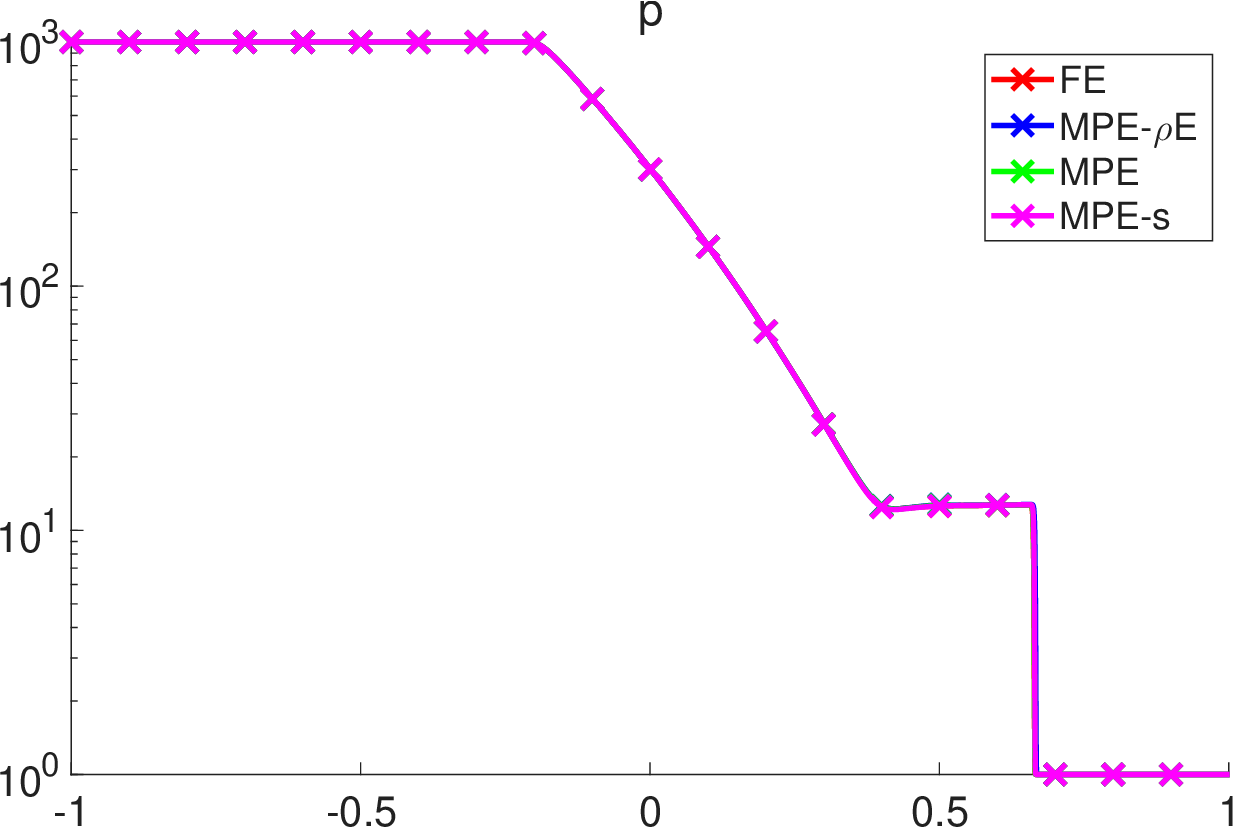}
\hfill
	\includegraphics[width=0.49\textwidth]{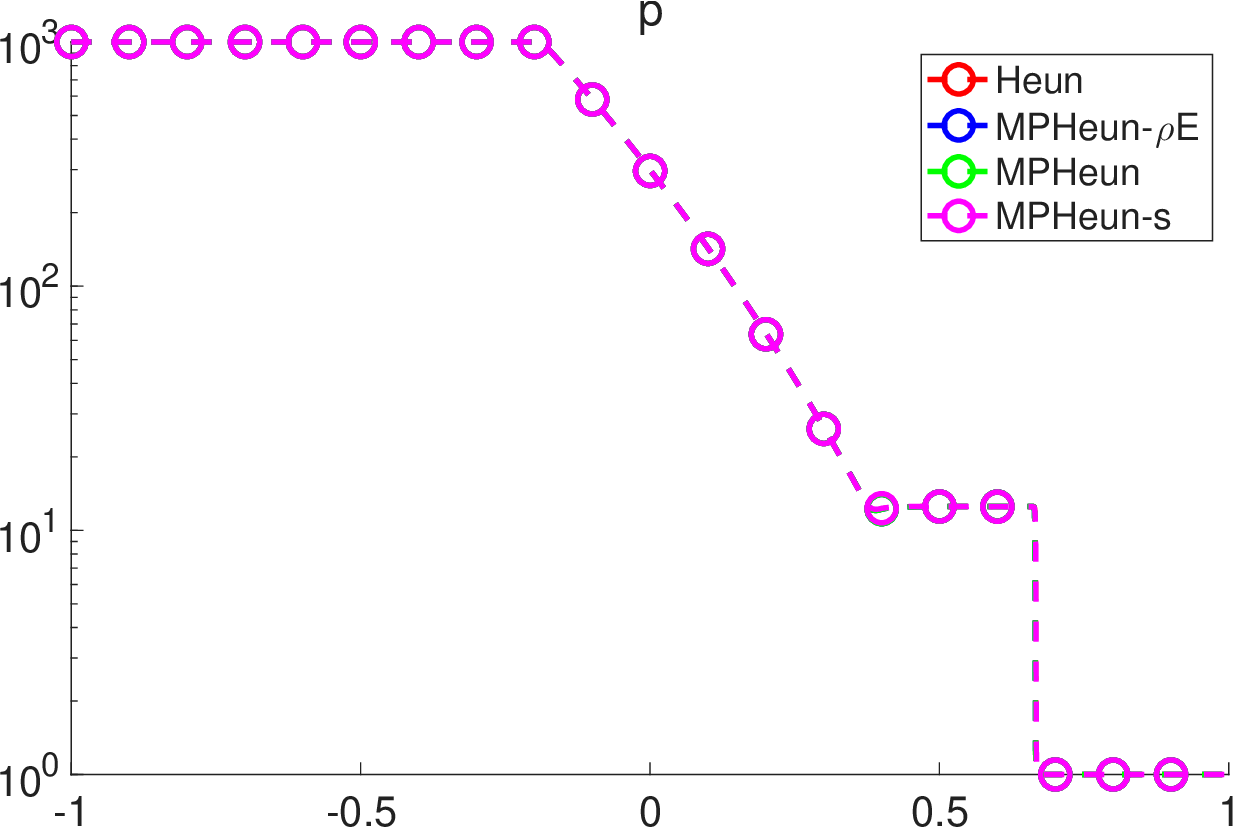}
\caption{Plots of the numerical solution of multi-species Euler equations \eqref{eq:EulerEquations}-\eqref{eq:multiEuler} with source \eqref{eq:source} using $\delta=10^4$, $N=4000$ cells as well as a CFL and safety factor as given in Table~\ref{tab:react_euler}. First-order methods are at the left,  second-order methods are at the right.}\label{fig:react_euler}
\end{figure}
As shown in Table~\ref{tab:react_euler}, the source terms are not directly included in the computation of the time step $\Delta t$. 
Instead, we select the maximal $\Delta t$ based only on the fluxes and the CFL number.
We choose a CFL that ensures the positivity of $\rho_i$, $p$, and $\rho E$. 
The CFL values in Table~\ref{tab:react_euler} are "maximal" in the sense that increasing the last significant digit results in negative densities, pressure, or total energy. 
We then run the simulations with that CFL number multiplied with a \textit{safety factor} specified in Table~\ref{tab:react_euler} to suppress spurious oscillations. 
We note that for the explicit schemes as well as for the second-order flux-balanced scheme we used a \textit{safety factor} of 1, whereas the other schemes have a \textit{safety factor} of 0.4 to mitigate the development of oscillations near the rarefaction wave. 

Looking at the runtime in Table~\ref{tab:react_euler}, even if the MPRK schemes have to solve many linear systems, they are favorable over the explicit schemes as they are faster. 
However, the second-order flux-balanced version with a runtime of around 2 minutes is the best performing scheme while the other second-order schemes range from 25 minutes to almost 90 minutes.

Despite these differences in computational speed, the numerical approximations are nearly indistinguishable in the visualizations (see Figure~\ref{fig:react_euler}) for the same order of accuracy. We notice better approximations of the contact discontinuity for all second order methods. 


\subsection{Contact Discontinuity}
\begin{table}\caption{Comparison of numerical schemes applied to one species Euler equations \eqref{eq:EulerEquations}-\eqref{eq:monoEuler} with IC \eqref{eq:cont_disc}  (Contact Discontinuity) and $N=1000$: stable CFL and runtime  in seconds}\label{tab:cont_disc}
\begin{tabular}{|llll|llll|}\hline
	\multicolumn{4}{|c|}{First order}&\multicolumn{4}{c|}{Second order} \\\hline
	Method & CFL & safety factor & Runtime & Method & CFL & safety factor & Runtime \\
	\hline 
	FE & 1 & 0.7 & 5.12			       & Heun            & 1 & 0.7 & 8.88  \\  
	MPE-$\rho E$ & 0.28 & 0.7 & 15.79    & MPHeun-$\rho E$ & 0.39 & 0.7 & 29.82  \\ 
	MPE & 0.38 & 0.7 & 10.70             & MPHeun          & 0.42 & 0.7 & 10.70   \\  
	MPE-s & 1 & 0.7 & 6.95 			   & MPHeun-s        & 1 & 0.7 & 17.60  \\  
	\hline 
\end{tabular}
\end{table}
\begin{figure}
		\includegraphics[width=0.49\textwidth]{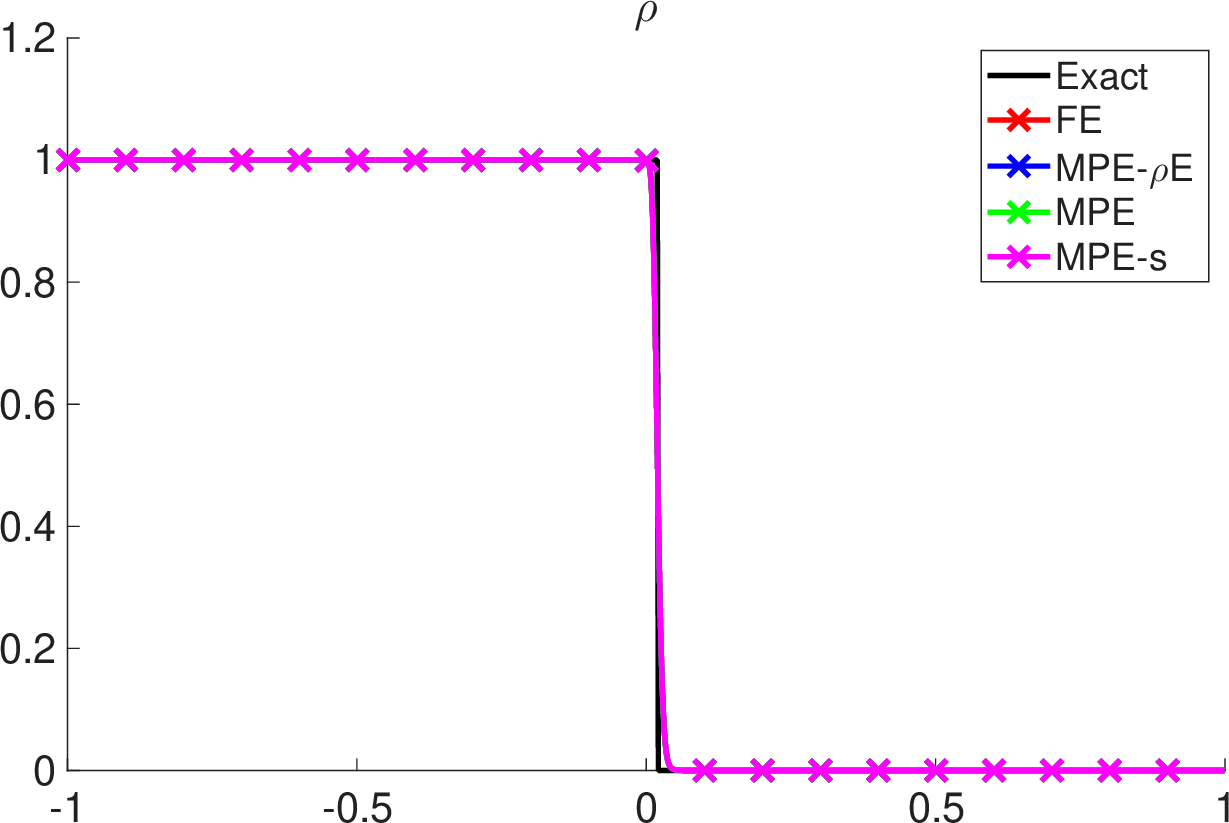}
	\hfill
		\includegraphics[width=0.49\textwidth]{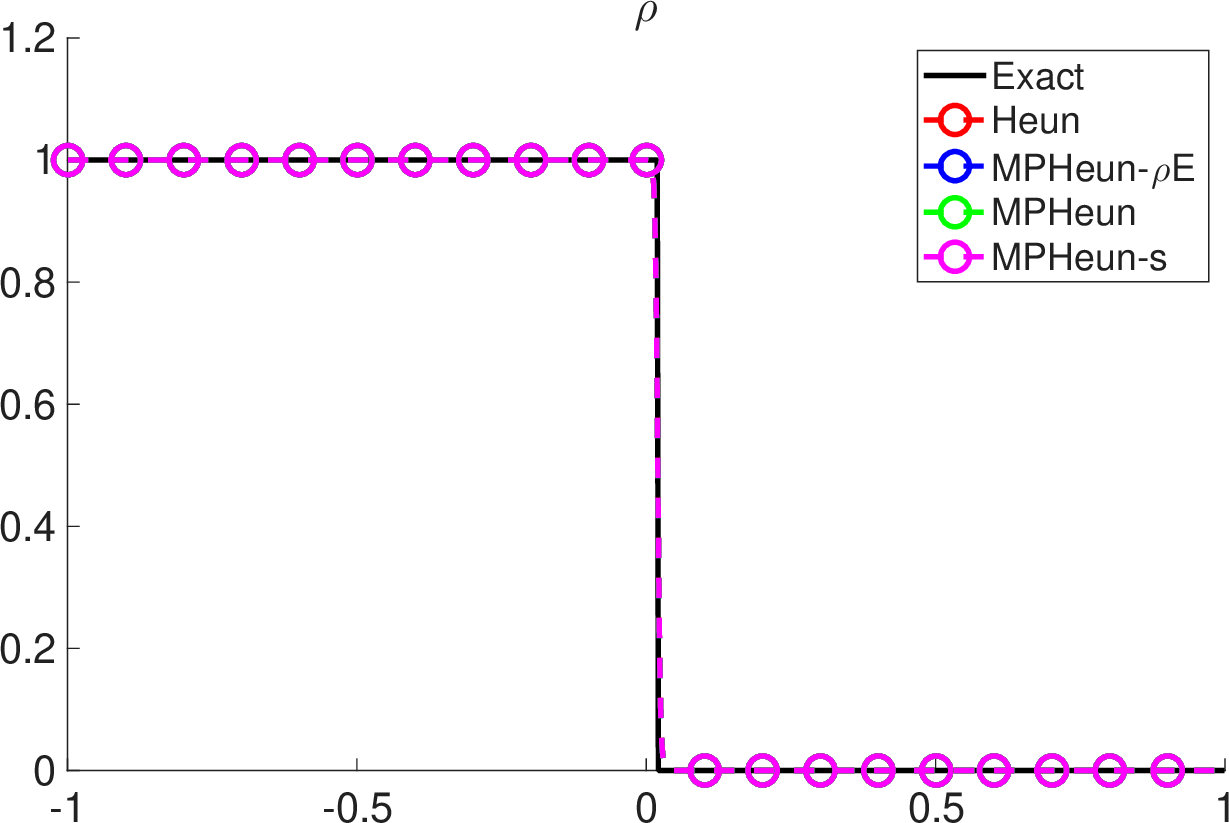}
	\\[3mm]
		\includegraphics[width=0.49\textwidth]{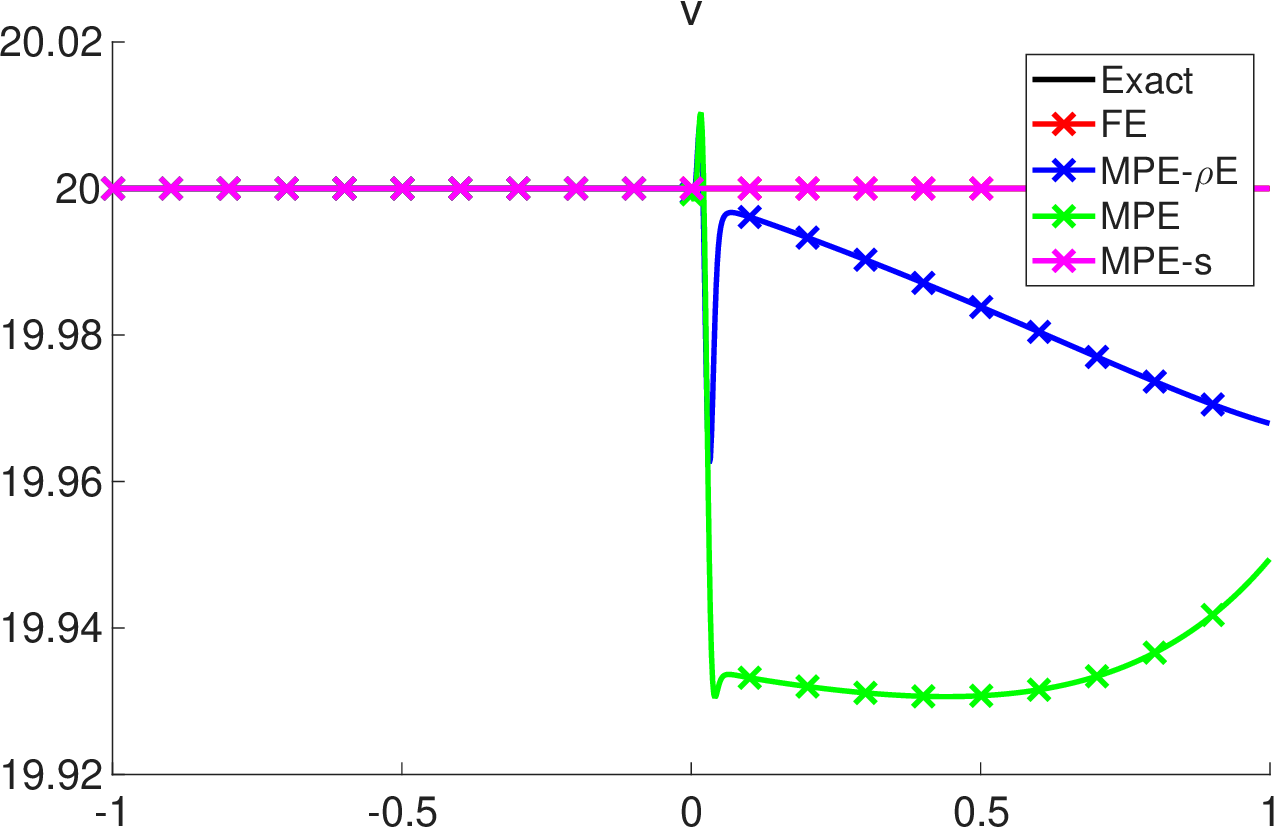}
		\hfill
	\includegraphics[width=0.49\textwidth]{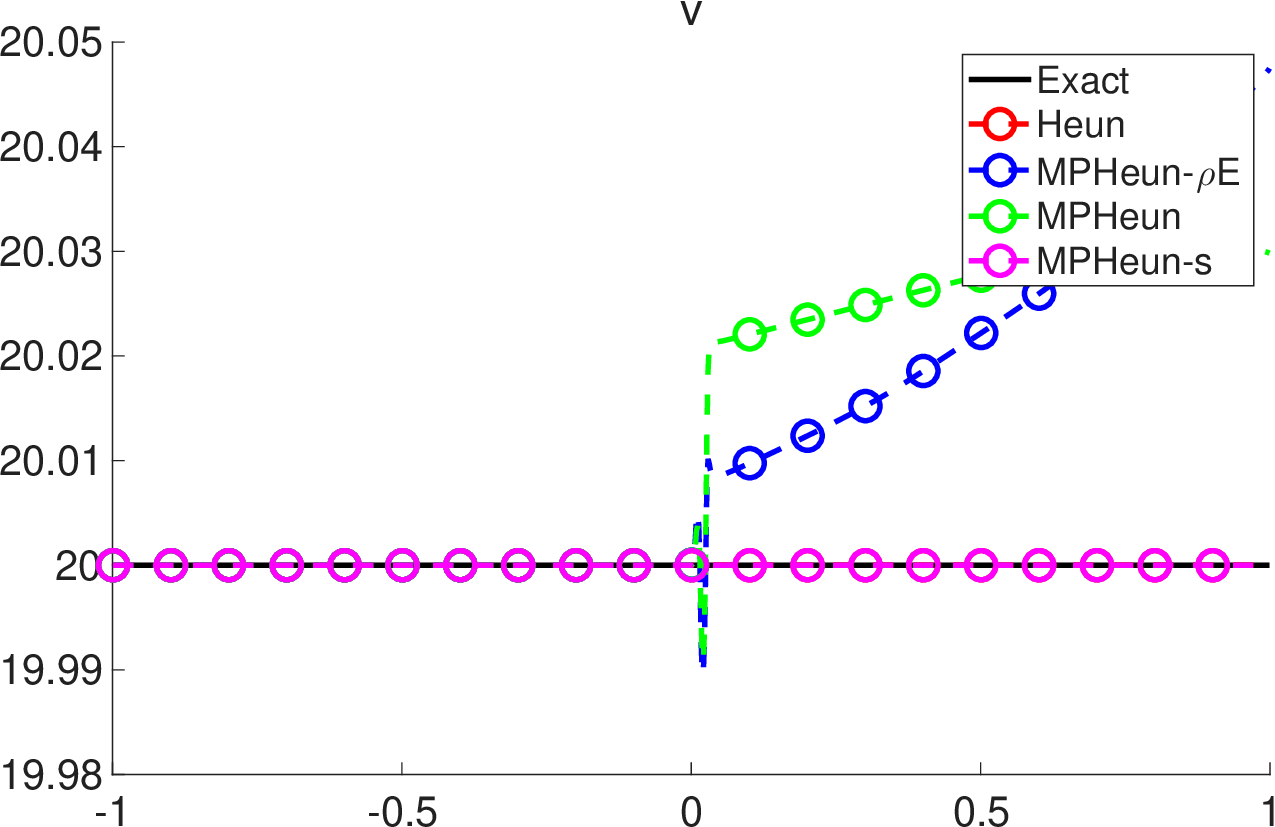}
	\\[3mm]
		\includegraphics[width=0.49\textwidth]{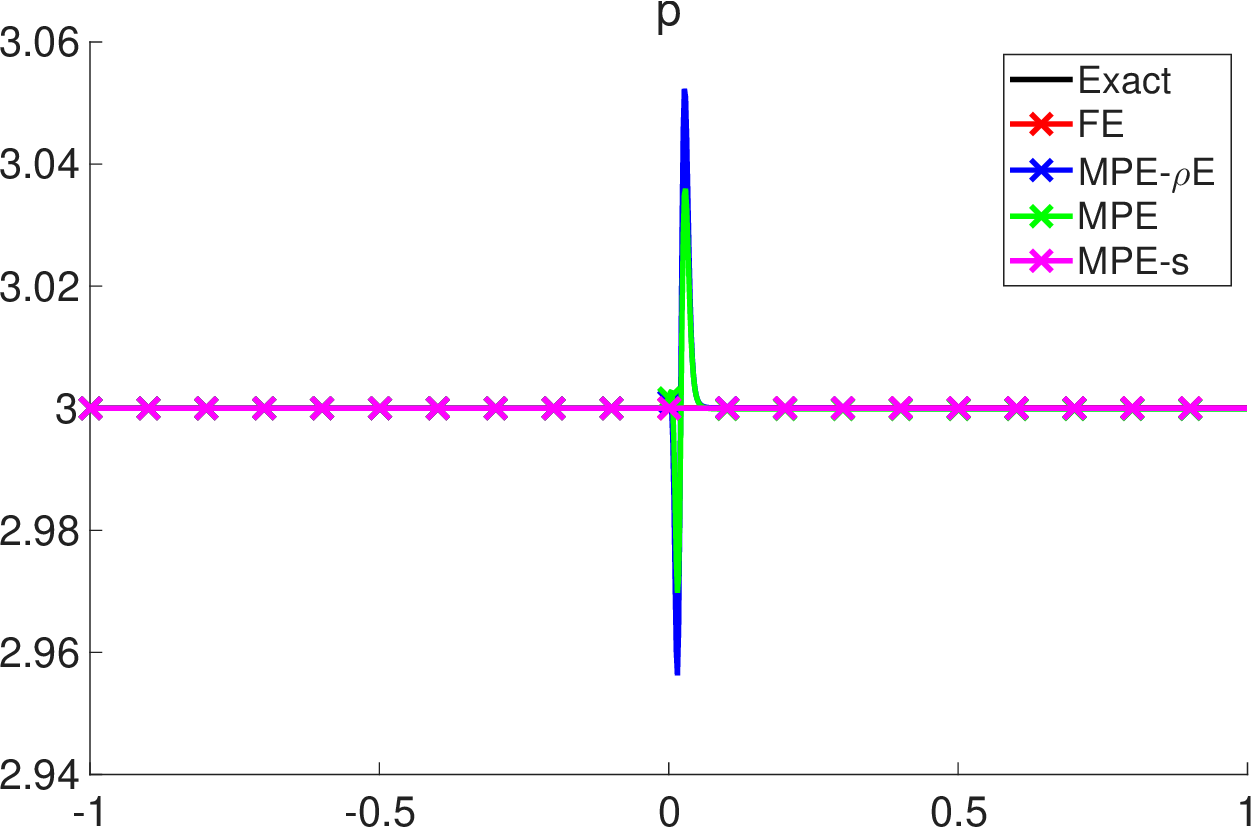}
	\hfill 	
		\includegraphics[width=0.49\textwidth]{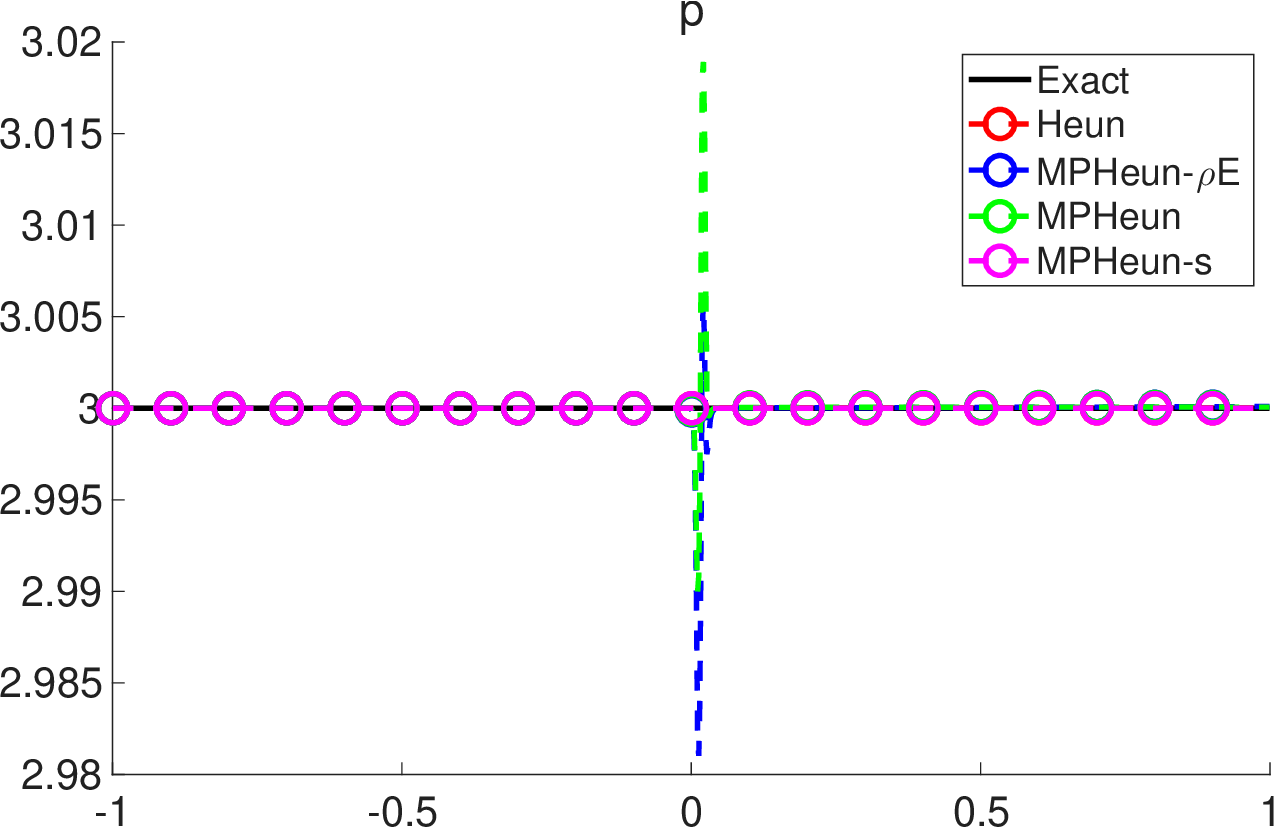}
	
	\caption{Plots of the numerical solution of Euler equations \eqref{eq:EulerEquations}-\eqref{eq:monoEuler} with contact discontinuity IC  \eqref{eq:cont_disc} with $N=1000$ using a CFL and safety factor given in Table~\ref{tab:cont_disc}. First-order methods are left,  second-order methods are on the right.}\label{fig:cont_disc}
\end{figure}
The contact discontinuity test case, we will consider, is a variation of the 1-2-3 Problem \cite{einfeldt1991godunov,toro2013riemann}, where the density jump and the velocity are much larger. It is a Riemann problem for the Euler equations with 
\begin{equation}\label{eq:cont_disc}
\b U_L=\begin{pmatrix} 1\\20\\3\end{pmatrix},\quad \b U_R=\begin{pmatrix} 10^{-6}\\ 20\\3\end{pmatrix}\end{equation}
as initial condition for the primitive variables $(\rho,u,p)$.

The advantage of the flux-balanced version over the classical MPRK schemes becomes evident when comparing the schemes to the test case with the contact discontinuity, \ie \eqref{eq:EulerEquations}-\eqref{eq:monoEuler}, with \eqref{eq:cont_disc}. Here, the classical MPRK schemes failed at maintaining constant velocity and pressure, see Figure~\ref{fig:cont_disc}.
In contrast, the flux-balanced version was able to preserve the constant states.  
Comparing the first-order schemes, we note that the LLF scheme is positivity-preserving for CFL=0.5, which means that the Patankar-trick is not needed for this test case. 

\subsection{Vacuum}
Finally, we challenge the flux-balanced version with a test leading to a vacuum.
It is another Riemann problem with a double rarefaction wave forming from the center. The Riemann problem is defined as 
\begin{equation}\label{eq:vacuum}
\b U_L=\begin{pmatrix} 1\\-20\\0.4\end{pmatrix},\quad \b U_R=\begin{pmatrix} 1\\ 20\\0.4\end{pmatrix}\end{equation}
in primitive variables. The final time is $\tend=0.03$. To obtain the exact solution, we use the exact Riemann solver for the Euler equations \cite{toro2013riemann} available at \cite{repoRPsolver}.
In Table~\ref{tab:vacuum}, we see that the flux-balanced version is comparable to the explicit variants in terms of runtime, but it is much faster than the MPRK schemes. 
Additionally, they show similar numerical results, see Figure~\ref{fig:vacuum}, while the MPRK schemes show nonphysical oscillations, which resulted in negative approximations for the pressure with a CFL number greater or equal to 0.02.
\begin{remark}
 Using a CFL number of 0.5 and setting $\rho_R=10^{-2}$, we observe that the flux-balanced MPE scheme yields a negative pressure after one time step. Hence, more research is needed in this field to provably control the pressure.
 \end{remark}
 \begin{table}\caption{Comparison of numerical schemes applied to Euler equations \eqref{eq:EulerEquations}-\eqref{eq:monoEuler} with IC \eqref{eq:vacuum} (Vacuum) and $N=1000$: stable CFL and runtime in seconds}\label{tab:vacuum}
\begin{tabular}{|llll|llll|}\hline
	\multicolumn{4}{|c|}{First order}&\multicolumn{4}{c|}{Second order} \\\hline
	Method & CFL & safety factor & Runtime & Method & CFL & safety factor & Runtime \\
	\hline 
	FE & 1 & 0.7 & 2.08	 & Heun & 1 & 0.7 & 2.51  \\  
	MPE-$\rho E$ & 0.01 & 0.7 & 141.69     & MPHeun-$\rho E$ & 0.10 & 0.7 & 32.14  \\ 
	MPE & 0.01 & 0.7 & 131.54      & MPHeun & 0.13 & 0.7 & 24.83   \\  
	MPE-s & 1 & 0.7 & 2.18 			   & MPHeun-s  & 1 & 0.7 & 5.24 \\  
	\hline 
\end{tabular}
 \end{table}
  \begin{figure}
 	\includegraphics[width=0.49\textwidth]{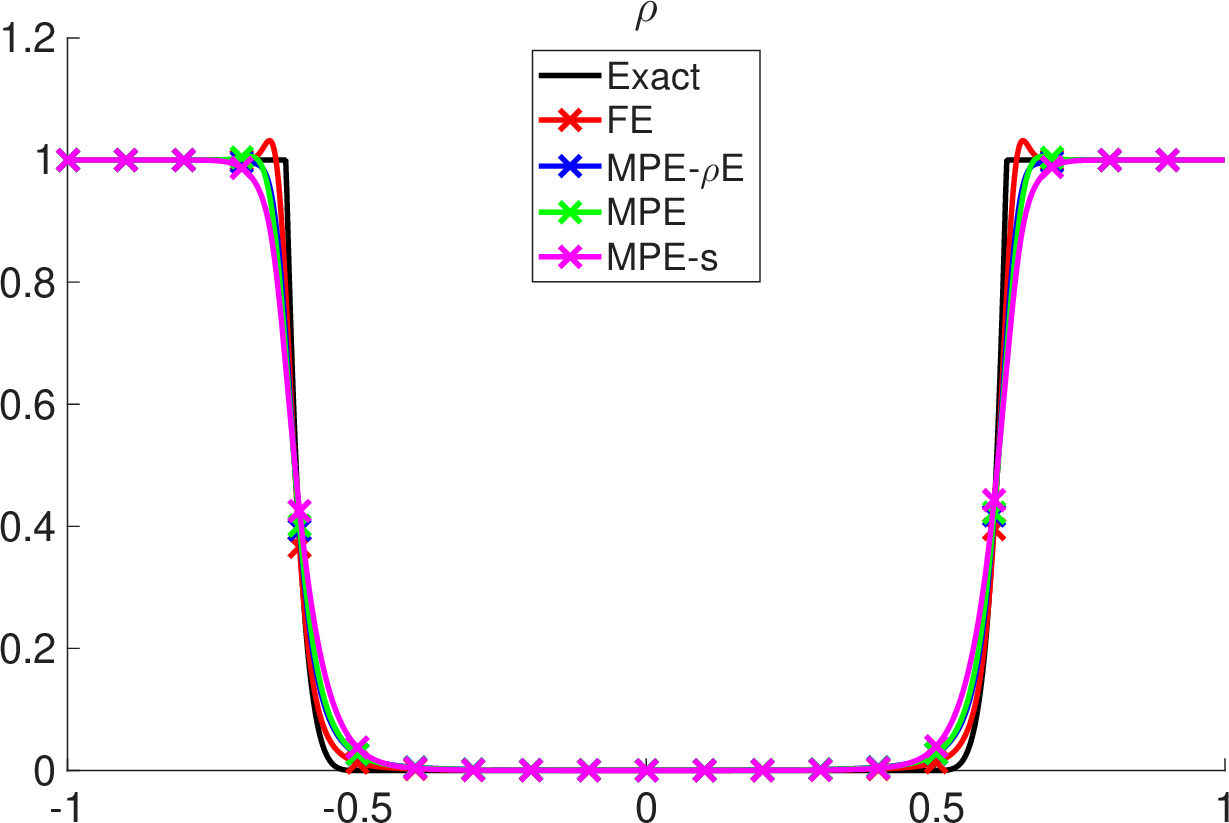}
 \hfill
 		\includegraphics[width=0.49\textwidth]{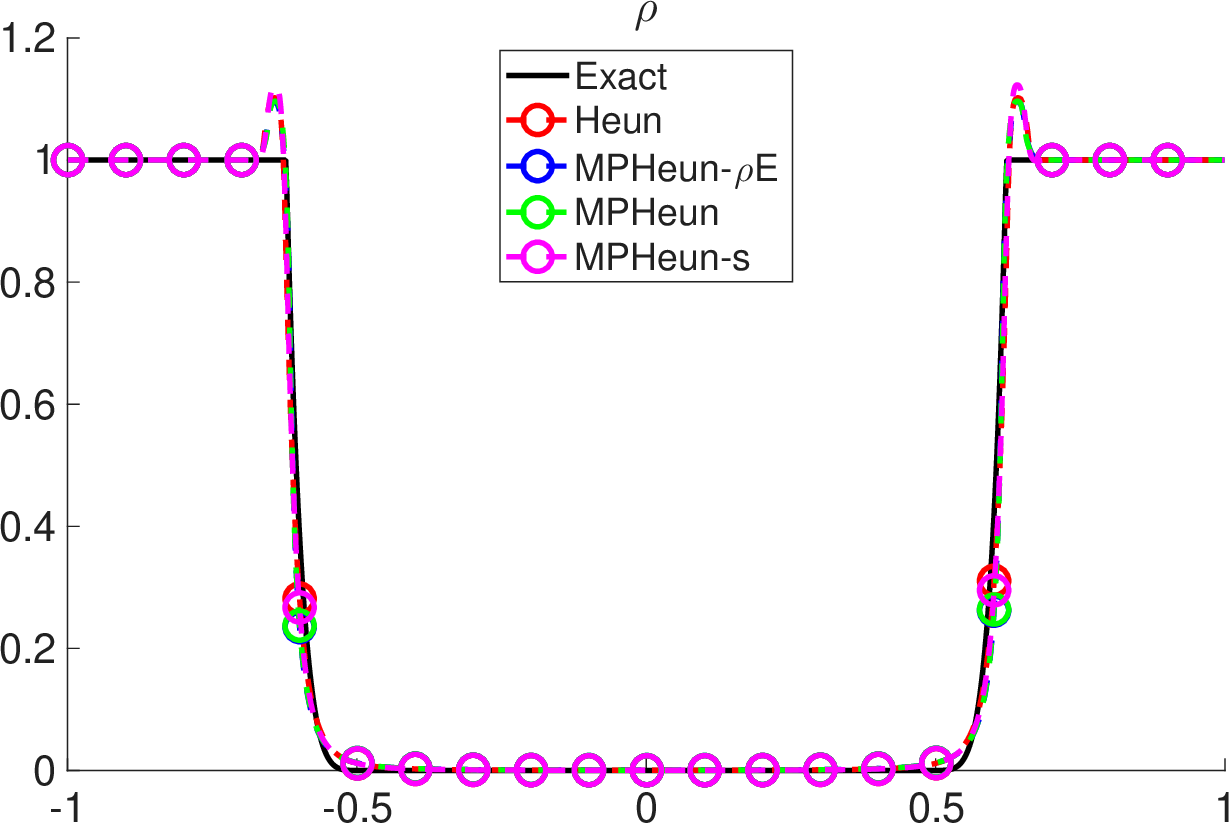}
 	\\
 	\includegraphics[width=0.49\textwidth]{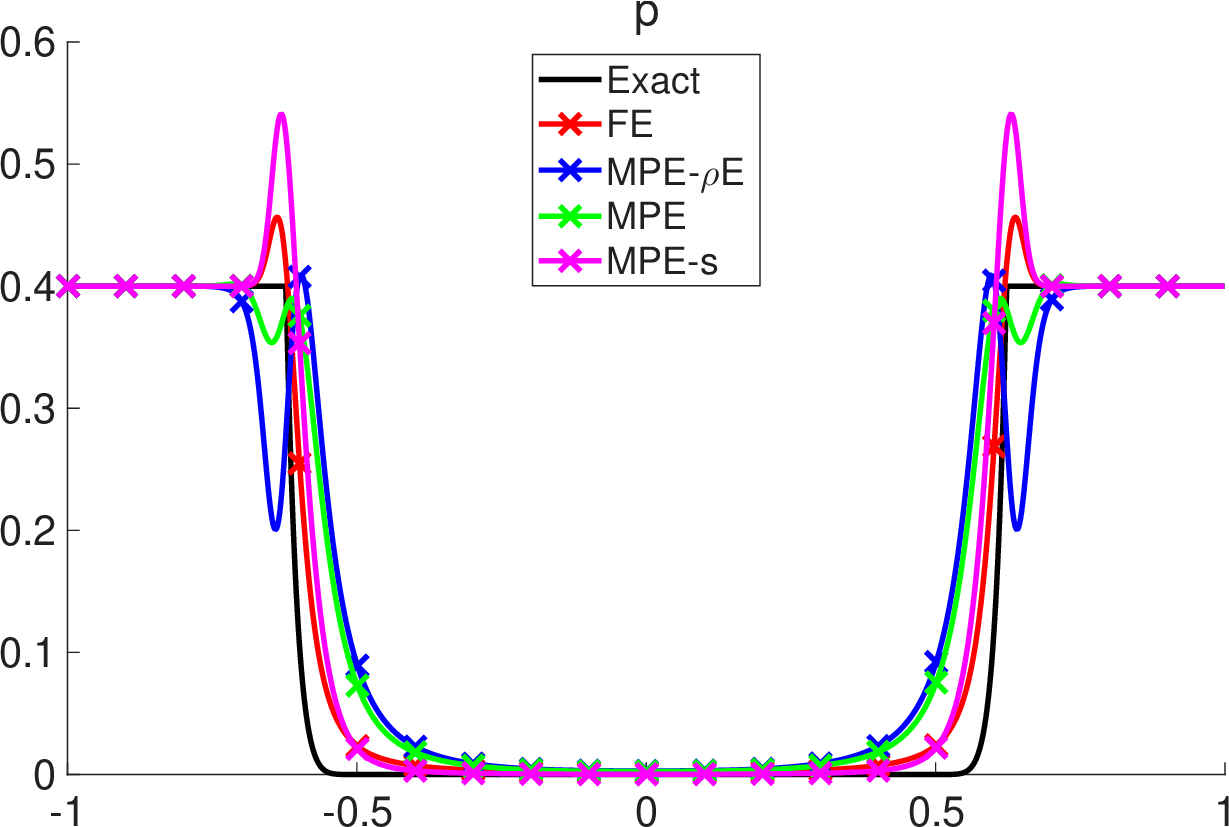}
 \hfill 	
 		\includegraphics[width=0.49\textwidth]{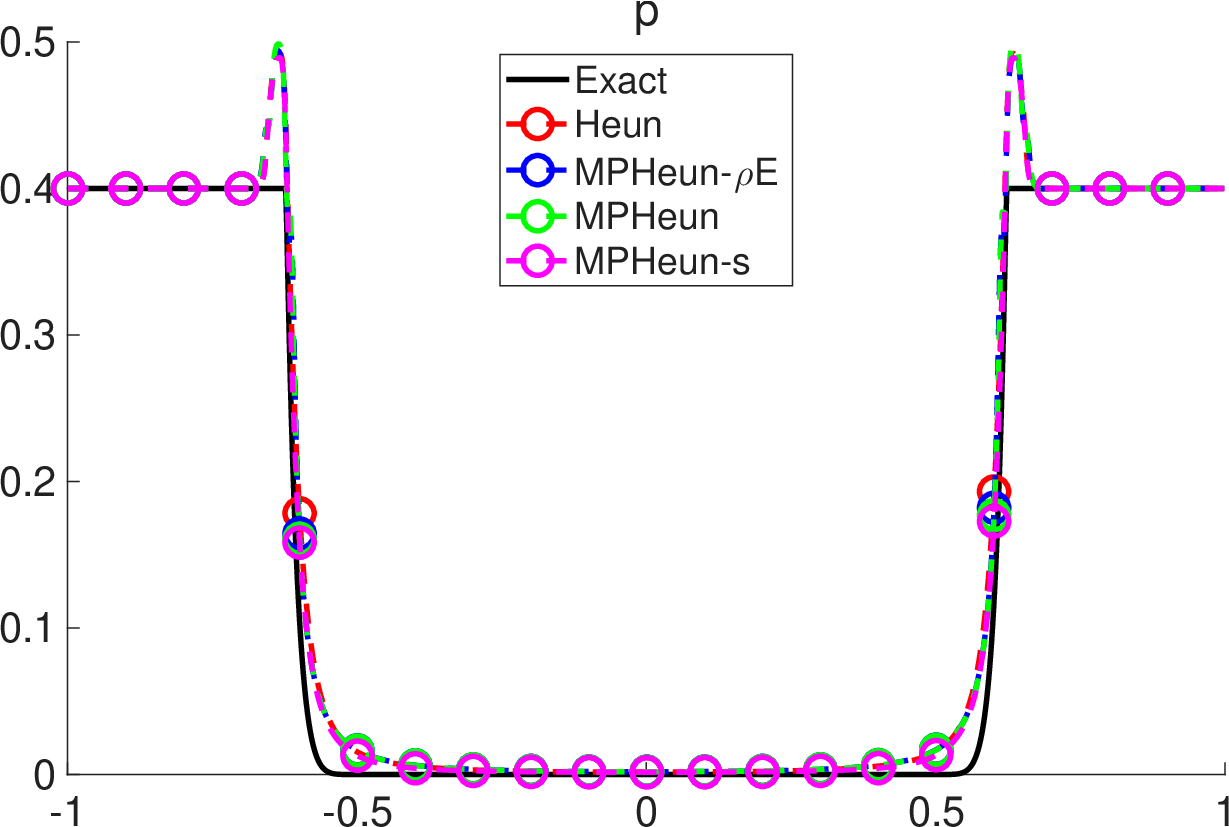}
 	 	\caption{Plots of the numerical solution of Euler equations \eqref{eq:EulerEquations}-\eqref{eq:monoEuler} with vacuum IC \eqref{eq:vacuum} with $N=1000$, and with CFL and safety factor given in Table~\ref{tab:vacuum}. First-order methods are left,  second-order methods are on the right.}\label{fig:vacuum}
 \end{figure}

\section{Summary and Outlook}\label{sec:summary}
In this work, we showed how to embed an MPRK scheme into a Finite Volume framework for Euler equations of gas dynamics, also taking stiff source terms into consideration. Furthermore, we showed that the MP-trick should not be applied blindly to all components of the balance law that should stay positive, but that it is instead of interest to use MP only for the densities and to explicitly weigh in a consistent manner also the other numerical fluxes to be consistent with the contact discontinuities. We performed numerical experiments confirming the expected order of accuracy and the performance of the different numerical schemes when applied to various test cases. The results show that the flux-balanced version of the MPRK is much more stable than the classical MPRK schemes for large CFLs.

Future works include the investigation of a hybrid variant of the explicit and flux-balanced schemes including further Patankar-type schemes, and the application of the flux-balanced schemes to further hyperbolic systems such as the shallow water equations. We will also consider in future works the combination of the MP trick on the energy equation and the simultaneous preservation of contact discontinuities.
Finally, our future work will also involve studying algorithms that are capable of preserving entropy properties 
and positivity simultaneously. We will investigate the algorithm's capability of controlling the positivity of
internal energy or
pressure via the EoS, and additionally, we will pursue the approach of solving the Euler system in primitive 
variables to evolve the internal energy or pressure directly while keeping conservation.
\section*{Statements and Declarations}
\bmhead{Acknowledgments}
The author T.\ Izgin gratefully acknowledges the financial support by Fulbright Germany. 
C.-W. Shu acknowledges partial support from NSF grant DMS-2309249.
D. Torlo is member of the GNCS group of INdAM.
\bmhead{Conflict of interest}
 On behalf of all authors, the corresponding author states that there is no
 conflict of interest.
 \vspace{1cm}

\bibliography{cas-refs}

@article{WSYS2009,
title = {High-order well-balanced schemes and applications to non-equilibrium flow},
journal = {Journal of Computational Physics},
volume = {228},
number = {18},
pages = {6682-6702},
year = {2009},
issn = {0021-9991},
doi = {https://doi.org/10.1016/j.jcp.2009.05.028},
url = {https://www.sciencedirect.com/science/article/pii/S0021999109002836},
author = {Wei Wang and Chi-Wang Shu and H.C. Yee and Björn Sjögreen}
}

@misc{repoRPsolver,
	title        = {{exact Riemann solver for compressible Euler equations (code)}},
	author       = {Yue Wu},
	year         = 2024,
	month        = {September},
	howpublished = {\url{https://github.com/YueWu2002/Euler_exact_Riemann}}
}

@article{lax_systems_1960,
	title        = {Systems of conservation laws},
	author       = {Lax, Peter and Wendroff, Burton},
	year         = 1960,
	month        = may,
	journal      = {Comm. Pure Appl. Math.},
	volume       = 13,
	number       = 2,
	pages        = {217--237},
	doi          = {10.1002/cpa.3160130205},
	issn         = {00103640, 10970312},
	language     = {en}
}

@article{EG2023,
	title        = {Invariant-domain preserving high-order time stepping: {II}. {IMEX} schemes},
	author       = {Ern, Alexandre and Guermond, Jean-Luc},
	year         = 2023,
	journal      = {SIAM J. Sci. Comput.},
	volume       = 45,
	number       = 5,
	pages        = {A2511--A2538},
	doi          = {10.1137/22M1505025},
	issn         = {1064-8275,1095-7197},
	url          = {https://doi.org/10.1137/22M1505025},
	fjournal     = {SIAM Journal on Scientific Computing},
	mrclass      = {65M60 (35L65 65M12)},
	mrnumber     = 4644400,
	mrreviewer   = {Victor\ Michel-Dansac}
}

@article{EG2022,
	title        = {Invariant-domain-preserving high-order time stepping: {I}. explicit {R}unge-{K}utta schemes},
	author       = {Ern, Alexandre and Guermond, Jean-Luc},
	year         = 2022,
	journal      = {SIAM J. Sci. Comput.},
	volume       = 44,
	number       = 5,
	pages        = {A3366--A3392},
	doi          = {10.1137/21M145793X},
	issn         = {1064-8275,1095-7197},
	url          = {https://doi.org/10.1137/21M145793X},
	fjournal     = {SIAM Journal on Scientific Computing},
	mrclass      = {65M60 (35L65 65M12 65N30)},
	mrnumber     = 4496709
}

@article{WS2023,
	title        = {Geometric quasilinearization framework for analysis and design of bound-preserving schemes},
	author       = {Wu, Kailiang and Shu, Chi-Wang},
	year         = 2023,
	journal      = {SIAM Rev.},
	volume       = 65,
	number       = 4,
	pages        = {1031--1073},
	doi          = {10.1137/21M1458247},
	issn         = {0036-1445,1095-7200},
	url          = {https://doi.org/10.1137/21M1458247},
	fjournal     = {SIAM Review},
	mrclass      = {65M12 (35L65 65M06 65M08 65M60)},
	mrnumber     = 4664482,
	mrreviewer   = {Jingwei\ Li}
}

@article{Dzanic2024a,
	title        = {Continuously bounds-preserving discontinuous {G}alerkin methods for hyperbolic conservation laws},
	author       = {Dzanic, T.},
	year         = 2024,
	journal      = {J. Comput. Phys.},
	volume       = 508,
	pages        = {Paper No. 113010, 25},
	doi          = {10.1016/j.jcp.2024.113010},
	issn         = {0021-9991,1090-2716},
	url          = {https://doi.org/10.1016/j.jcp.2024.113010},
	fjournal     = {Journal of Computational Physics},
	mrclass      = {65M60 (35L40 35L65)},
	mrnumber     = 4735178,
	mrreviewer   = {Xiong\ Meng}
}

@article{Dzanic2024b,
	title        = {A note on higher-order and nonlinear limiting approaches for continuously bounds-preserving discontinuous {G}alerkin methods},
	author       = {Dzanic, T.},
	year         = 2024,
	journal      = {J. Comput. Phys.},
	volume       = 516,
	pages        = {Paper No. 113367, 7},
	doi          = {10.1016/j.jcp.2024.113367},
	issn         = {0021-9991,1090-2716},
	url          = {https://doi.org/10.1016/j.jcp.2024.113367},
	fjournal     = {Journal of Computational Physics},
	mrclass      = {65M60 (65M20)},
	mrnumber     = 4791328,
	mrreviewer   = {Yong\ Liu}
}

@article{GD2021,
	title        = {A posteriori subcell finite volume limiter for general {$P_NP_M$} schemes: applications from gasdynamics to relativistic magnetohydrodynamics},
	author       = {Gaburro, Elena and Dumbser, Michael},
	year         = 2021,
	journal      = {J. Sci. Comput.},
	volume       = 86,
	number       = 3,
	pages        = {Paper No. 37, 41},
	doi          = {10.1007/s10915-020-01405-8},
	issn         = {0885-7474,1573-7691},
	url          = {https://doi.org/10.1007/s10915-020-01405-8},
	fjournal     = {Journal of Scientific Computing},
	mrclass      = {65M08 (65M50 76N15 76W05 76Y05)},
	mrnumber     = 4205083,
	mrreviewer   = {Pedro\ Gonz\'alez-Casanova}
}

@article{RPG2022,
	title        = {Subcell limiting strategies for discontinuous {G}alerkin spectral element methods},
	author       = {Rueda-Ram\'irez, Andr\'es M. and Pazner, Will and Gassner, Gregor J.},
	year         = 2022,
	journal      = {Comput. \& Fluids},
	volume       = 247,
	pages        = {Paper No. 105627, 19},
	doi          = {10.1016/j.compfluid.2022.105627},
	issn         = {0045-7930,1879-0747},
	url          = {https://doi.org/10.1016/j.compfluid.2022.105627},
	fjournal     = {Computers \& Fluids. An International Journal},
	mrclass      = {65M60 (35L65)},
	mrnumber     = 4475945
}

@article{CGMPT2023,
	title        = {Robust second-order approximation of the compressible {E}uler equations with an arbitrary equation of state},
	author       = {Clayton, Bennett and Guermond, Jean-Luc and Maier, Matthias and Popov, Bojan and Tovar, Eric J.},
	year         = 2023,
	journal      = {J. Comput. Phys.},
	volume       = 478,
	pages        = {Paper No. 111926, 21},
	doi          = {10.1016/j.jcp.2023.111926},
	issn         = {0021-9991,1090-2716},
	url          = {https://doi.org/10.1016/j.jcp.2023.111926},
	fjournal     = {Journal of Computational Physics},
	mrclass      = {76N10 (65M60)},
	mrnumber     = 4543557,
	mrreviewer   = {Alberto\ Valli}
}

@article{PS1996,
	title        = {On positivity preserving finite volume schemes for {E}uler equations},
	author       = {Perthame, Benoit and Shu, Chi-Wang},
	year         = 1996,
	journal      = {Numer. Math.},
	volume       = 73,
	number       = 1,
	pages        = {119--130},
	doi          = {10.1007/s002110050187},
	issn         = {0029-599X,0945-3245},
	url          = {https://doi.org/10.1007/s002110050187},
	fjournal     = {Numerische Mathematik},
	mrclass      = {65M60 (76M25 76N15)},
	mrnumber     = 1379283,
	mrreviewer   = {Michael\ Fr\"ohner}
}

@techreport{ricchiuto2011habilitation,
	title        = {Contributions to the development of residual discretizations for hyperbolic conservation laws with application to shallow water flows},
	author       = {{Ricchiuto}, Mario},
	year         = 2011,
	month        = 12,
	address      = {Bordeaux, France},
	institution  = {Universit{\'e} Sciences et Technologies-Bordeaux I},
	type         = {Habilitation thesis}
}

@article{HR2020,
	title        = {Space-time residual distribution on moving meshes},
	author       = {Hubbard, M. E. and {Ricchiuto}, M. and S\'arm\'any, D.},
	year         = 2020,
	journal      = {Comput. Math. Appl.},
	volume       = 79,
	number       = 5,
	pages        = {1561--1589},
	doi          = {10.1016/j.camwa.2019.09.019},
	issn         = {0898-1221,1873-7668},
	url          = {https://doi.org/10.1016/j.camwa.2019.09.019},
	fjournal     = {Computers \& Mathematics with Applications. An International Journal},
	mrclass      = {65M08 (65M50 76B15)},
	mrnumber     = 4065801,
	mrreviewer   = {Jean-Pierre\ Croisille}
}

@article{IMPV2025,
	title        = {Modified Patankar Linear Multistep methods for production-destruction systems},
	author       = {Izzo, Giuseppe and Messina, Eleonora and Pezzella, Mario and Vecchio, Antonia},
	year         = 2025,
	journal      = {Journal of Scientific Computing},
	publisher    = {Springer},
	volume       = 102,
	number       = 3,
	pages        = 87
}

@article{SSPMPRK2,
	title        = {Positivity-preserving time discretizations for production-destruction equations with applications to non-equilibrium flows},
	author       = {Huang, J.  and Shu, C.-W.},
	year         = 2019,
	journal      = {J. Sci. Comput.},
	volume       = 78,
	number       = 3,
	pages        = {1811--1839},
	doi          = {},
	issn         = {0885-7474},
	url          = {https://doi.org/10.1007/s10915-018-0852-1},
	fjournal     = {Journal of Scientific Computing},
	mrclass      = {65M06 (65L06 65M20 76V05)},
	mrnumber     = 3934688,
	mrreviewer   = {Qifeng Zhang}
}

@article{sandu2001positive,
	title        = {Positive numerical integration methods for chemical kinetic systems},
	author       = {Sandu, A.},
	year         = 2001,
	journal      = {J. Comput. Phys.},
	volume       = 170,
	number       = 2,
	pages        = {589--602},
	doi          = {10.1006/jcph.2001.6750},
	issn         = {0021-9991},
	url          = {https://doi.org/10.1006/jcph.2001.6750},
	fjournal     = {Journal of Computational Physics},
	mrclass      = {80M20 (65L99 80A30)},
	mrnumber     = 1844904
}

@article{MCD2020,
	title        = {{{G}e{C}o}: {G}eometric {C}onservative nonstandard schemes for biochemical systems},
	author       = {Martiradonna, Angela and Colonna, Gianpiero and Diele, Fasma},
	year         = 2020,
	journal      = {Appl. Numer. Math.},
	volume       = 155,
	pages        = {38--57},
	doi          = {10.1016/j.apnum.2019.12.004},
	issn         = {0168-9274},
	url          = {https://doi.org/10.1016/j.apnum.2019.12.004},
	fjournal     = {Applied Numerical Mathematics. An IMACS Journal},
	mrclass      = {92C40 (65P10)},
	mrnumber     = 4087156
}

@article{NSARK,
	title        = {Order conditions for {Runge--Kutta}-like methods with solution-dependent coefficients},
	author       = {Thomas {{Izgin}} and David I. Ketcheson and Andreas Meister},
	year         = 2025,
	journal      = {Communications in Applied Mathematics and Computational Science},
	volume       = {20-1},
	pages        = {29--66},
	doi          = {10.2140/camcos.2025.20.29},
	url          = {https://doi.org/10.2140/camcos.2025.20.29}
}

@phdthesis{IzginThesis,
	title        = {A Unifying Theory for {Runge-Kutta}-like Time Integrators: Convergence and Stability},
	author       = {Izgin, Thomas},
	year         = {2024},
	school       = {University of Kassel},
	doi          = {10.17170/kobra-202402059522}
}

@article{KM18,
	title        = {On order conditions for modified {P}atankar-{R}unge-{K}utta schemes},
	author       = {Kopecz, S. and Meister, A.},
	year         = 2018,
	journal      = {Appl. Numer. Math.},
	publisher    = {Elsevier},
	volume       = 123,
	pages        = {159--179},
	issn         = {0168-9274},
	url          = {https://doi.org/10.1016/j.apnum.2017.09.004},
	fjournal     = {Applied Numerical Mathematics. An IMACS Journal},
	mrclass      = {65L06},
	mrnumber     = 3711996,
	mrreviewer   = {D. Shirokoff}
}

@article{IR2023,
	title        = {{Using Machine Learning to Design Time Step Size Controllers for Stable Time Integrators}},
      author={Thomas Izgin and Hendrik Ranocha},
      year={2025},
      eprint={2312.01796},
      archivePrefix={arXiv},
      primaryClass={math.NA},
      journal={https://arxiv.org/abs/2312.01796}, 
}

@article{IKM2022b,
	title        = {On the stability of unconditionally positive and linear invariants preserving time integration schemes},
	author       = {{{Izgin}}, Thomas and Kopecz, Stefan and Meister, Andreas},
	year         = 2022,
	journal      = {SIAM J. Numer. Anal.},
	volume       = 60,
	number       = 6,
	pages        = {3029--3051},
	doi          = {10.1137/22M1480318},
	issn         = {0036-1429,1095-7170},
	fjournal     = {SIAM Journal on Numerical Analysis},
	mrclass      = {65L05 (34D20 65L20)},
	mrnumber     = 4509959,
	mrreviewer   = {Lajos\ L\'{o}czi}
}

@article{AKM2020,
	title        = {A comprehensive theory on generalized {BBKS} schemes},
	author       = {\'{A}vila, Andr\'{e}s I. and Kopecz, Stefan and Meister, Andreas},
	year         = 2020,
	journal      = {Appl. Numer. Math.},
	volume       = 157,
	pages        = {19--37},
	doi          = {10.1016/j.apnum.2020.05.027},
	issn         = {0168-9274,1873-5460},
	fjournal     = {Applied Numerical Mathematics. An IMACS Journal},
	mrclass      = {65L06 (65L04 65L20)},
	mrnumber     = 4109346,
	mrreviewer   = {B\"{u}lent\ Karas\"{o}zen}
}

@article{KM2018b,
	title        = {{Unconditionally positive and conservative third order modified {P}atankar-{R}unge-{K}utta discretizations of production-destruction systems}},
	author       = {Kopecz, Stefan and Meister, Andreas},
	year         = 2018,
	journal      = {BIT},
	volume       = 58,
	number       = 3,
	pages        = {691--728},
	doi          = {10.1007/s10543-018-0705-1},
	issn         = {0006-3835,1572-9125},
	fjournal     = {BIT. Numerical Mathematics},
	mrclass      = {65L06 (65L05 65L20)},
	mrnumber     = 3855688,
	mrreviewer   = {Luigi\ Brugnano}
}

@article{BDM2003,
	title        = {A high-order conservative {P}atankar-type discretisation for stiff systems of production-destruction equations},
	author       = {Burchard, Hans and Deleersnijder, Eric and Meister, Andreas},
	year         = 2003,
	journal      = {Appl. Numer. Math.},
	volume       = 47,
	number       = 1,
	pages        = {1--30},
	doi          = {10.1016/S0168-9274(03)00101-6},
	issn         = {0168-9274,1873-5460},
	fjournal     = {Applied Numerical Mathematics. An IMACS Journal},
	mrclass      = {65L06 (65L05)},
	mrnumber     = 2003144,
	mrreviewer   = {Gabriela\ Schranz-Kirlinger}
}

@article{BIM2022,
	title        = {Positivity-preserving methods for ordinary differential equations},
	author       = {{Blanes, Sergio} and {Iserles, Arieh} and {Macnamara, Shev}},
	year         = 2022,
	journal      = {ESAIM: M2AN},
	volume       = 56,
	number       = 6,
	pages        = {1843--1870},
	doi          = {10.1051/m2an/2022042},
	url          = {https://doi.org/10.1051/m2an/2022042}
}

@article{Ciallella2022,
	title        = {An arbitrary high order and positivity preserving method for the shallow water equations},
	author       = {Ciallella, Mirco and Micalizzi, Lorenzo and \"{O}ffner, Philipp and Torlo, Davide},
	year         = 2022,
	journal      = {Comput. \& Fluids},
	volume       = 247,
	pages        = 105630,
	doi          = {10.1016/j.compfluid.2022.105630},
	issn         = {0045-7930,1879-0747},
	fjournal     = {Computers \& Fluids. An International Journal},
	mrclass      = {65M08 (76U60)},
	mrnumber     = 4480759
}

@article{Sandu2001,
	title        = {Positive numerical integration methods for chemical kinetic systems},
	author       = {Sandu, Adrian},
	year         = 2001,
	journal      = {J. Comput. Phys.},
	volume       = 170,
	number       = 2,
	pages        = {589--602},
	doi          = {10.1006/jcph.2001.6750},
	issn         = {0021-9991,1090-2716},
	fjournal     = {Journal of Computational Physics},
	mrclass      = {80M20 (65L99 80A30)},
	mrnumber     = 1844904
}

@misc{TORLO,
	title        = {A New Stability Approach for Positivity-Preserving {Patankar-type} Schemes},
	author       = {Torlo, D. and {\"O}ffner, P. and Ranocha, H.},
	year         = 2021,
	eprint       = {2108.07347},
	archiveprefix = {arXiv},
	primaryclass = {math.NA}
}

@article{IssuesMPRK,
	title        = {Issues with positivity-preserving {P}atankar-type schemes},
	author       = {Torlo, D. and {\"O}ffner, P. and Ranocha, H.},
	year         = 2022,
	journal      = {Appl. Numer. Math.},
	volume       = 182,
	pages        = {117--147},
	doi          = {10.1016/j.apnum.2022.07.014},
	issn         = {0168-9274},
	url          = {https://doi.org/10.1016/j.apnum.2022.07.014},
	fjournal     = {Applied Numerical Mathematics. An IMACS Journal},
	mrclass      = {65L06 (65L04)},
	mrnumber     = 4469089
}

@book{Patankar1980,
	title        = {Numerical heat transfer and fluid flow},
	author       = {Patankar, S. V.},
	year         = 1980,
	publisher    = {Hemisphere Pub. Corp. New York},
	address      = {Washington},
	series       = {{Series in computational methods in mechanics and thermal sciences}},
	isbn         = {0-07-048740-5},
	url          = {http://opac.inria.fr/record=b1085925}
}

@article{KM18Order3,
	title        = {Unconditionally positive and conservative third order modified {P}atankar-{R}unge-{K}utta discretizations of production-destruction systems},
	author       = {Kopecz, S. and Meister, A.},
	year         = 2018,
	journal      = {BIT},
	publisher    = {Springer},
	volume       = 58,
	number       = 3,
	pages        = {691--728},
	issn         = {0006-3835},
	url          = {https://doi.org/10.1007/s10543-018-0705-1},
	fjournal     = {BIT. Numerical Mathematics},
	mrclass      = {65L06 (65L05 65L20)},
	mrnumber     = 3855688,
	mrreviewer   = {Luigi Brugnano}
}

@article{MPDeC,
	title        = {Arbitrary high-order, conservative and positivity preserving {P}atankar-type deferred correction schemes},
	author       = {{\"O}ffner, P. and Torlo, D.},
	year         = 2020,
	journal      = {Appl. Numer. Math.},
	volume       = 153,
	pages        = {15--34},
	doi          = {},
	issn         = {0168-9274},
	url          = {https://doi.org/10.1016/j.apnum.2020.01.025},
	fjournal     = {Applied Numerical Mathematics. An IMACS Journal},
	mrclass      = {65L06},
	mrnumber     = 4064785,
	mrreviewer   = {Helmut Podhaisky}
}

@article{SSPMPRK3,
	title        = {A third-order unconditionally positivity-preserving scheme for production-destruction equations with applications to non-equilibrium flows},
	author       = {Huang, J.  and Zhao, W. and Shu, C.-W.},
	year         = 2019,
	journal      = {J. Sci. Comput.},
	volume       = 79,
	number       = 2,
	pages        = {1015--1056},
	doi          = {},
	issn         = {0885-7474},
	url          = {https://doi.org/10.1007/s10915-018-0881-9},
	fjournal     = {Journal of Scientific Computing},
	mrclass      = {65M06 (76Nxx 76V05)},
	mrnumber     = 3969000
}

@book{toro2013riemann,
  title={Riemann Solvers and Numerical Methods for Fluid Dynamics: A Practical Introduction},
  author={Toro, E.F.},
  isbn={9783540498346},
  lccn={2009921818},
  year={2009},
  publisher={Springer},
  address = {Berlin Heidelberg}
}

@article{einfeldt1991godunov,
  title={On {Godunov}-type methods near low densities},
  author={Einfeldt, Bernd and Munz, Claus-Dieter and Roe, Philip L and Sj{\"o}green, Bj{\"o}rn},
  journal={Journal of computational physics},
  volume={92},
  number={2},
  pages={273--295},
  year={1991},
  publisher={Elsevier}
}

@book{leveque2002finite,
 	address={Cambridge},
 	series={Cambridge Texts in Applied Mathematics},
 	title={Finite Volume Methods for Hyperbolic Problems},
 	publisher={Cambridge University Press},
 	author={LeVeque,Randall J.},
 	year={2002},
 	collection={Cambridge Texts in Applied Mathematics}
}

@incollection{kuzmin2005algebraic,
author="Kuzmin, Dmitri
and M{\"o}ller, Matthias",
editor="Kuzmin, Dmitri
and L{\"o}hner, Rainald
and Turek, Stefan",
title="Algebraic Flux Correction II. Compressible Euler Equations",
bookTitle="Flux-Corrected Transport: Principles, Algorithms, and Applications",
year="2005",
publisher="Springer Berlin Heidelberg",
address="Berlin, Heidelberg",
pages="207--250",
isbn="978-3-540-27206-9",
doi="10.1007/3-540-27206-2_7",
url="https://doi.org/10.1007/3-540-27206-2_7"
}

@article{gaburro2024discontinuous,
	title={Discontinuous {Galerkin} schemes for hyperbolic systems in non-conservative variables: quasi-conservative formulation with subcell finite volume corrections},
	author={Gaburro, Elena and Boscheri, Walter and Chiocchetti, Simone and Ricchiuto, Mario},
	journal={Computer Methods in Applied Mechanics and Engineering},
	volume={431},
	pages={117311},
	year={2024},
	publisher={Elsevier}
}

@article{abgrall2018high,
	title={A high-order nonconservative approach for hyperbolic equations in fluid dynamics},
	author={Abgrall, Remi and Bacigaluppi, P and Tokareva, S},
	journal={Computers \& Fluids},
	volume={169},
	pages={10--22},
	year={2018},
	publisher={Elsevier}
}

\end{document}